\documentclass[10pt]{amsart}
\usepackage[alphabetic,lite]{amsrefs}
\usepackage{amsmath,fullpage}
\usepackage{amsthm}
\usepackage{mathrsfs}

\usepackage{xcolor}

\usepackage{enumerate}

\usepackage{hyperref}
\hypersetup{
    colorlinks=true,       
    linkcolor=blue,          
    citecolor=magenta,        
    filecolor=green,      
    urlcolor=purple           
}

    \newtheorem{Lem}{Lemma}[section]
    \newtheorem{Prop}[Lem]{Proposition}
    \newtheorem{Thm}[Lem]{Theorem}
    \newtheorem{Cor}[Lem]{Corollary}

\theoremstyle{definition}

    \newtheorem{Def}[Lem]{Definition}
    \newtheorem{Exa}[Lem]{Example}
    \newtheorem{Rem}[Lem]{Remark}

\newcommand{\Pz}{\mathbb{P}^4_{\mathbb{Z}}}
\newcommand{\PR}{\mathbb{P}^4_{R}}
\newcommand{\Pqz}{\mathbb{P}^{14}_{\mathbb{Z}}}
\newcommand{\PqR}{\mathbb{P}^{14}_{R}}
\newcommand{\Pdzd}{{\mathbb{P}^{2}_{\mathbb{Z}}}^\vee}
\newcommand{\PdRd}{{\mathbb{P}^{2}_{R}}^\vee}

\newcommand{\Pqd}{{\mathbb{P}^{14}_{k}}^\vee}
\newcommand{\PdR}{\mathbb{P}^2_{R}}
\newcommand{\Pdz}{\mathbb{P}^2_{\mathbb{Z}}}

\newcommand{\PuR}{\mathbb{P}^1_{R}}
\newcommand{\Spec}{\textrm{Spec}\,}
\newcommand{\spz}{\Spec (\mathbb{Z})}
\newcommand{\spR}{\Spec (R)}
\newcommand{\F}{\mathscr{F}}
\newcommand{\sO}{\mathscr{O}}

\newcommand{\fl}{{\rm Fl}}
\newcommand{\bp}{\mathbb{P}^2_k}
\newcommand{\bpd}{{\mathbb{P}^2_k}^\vee}
\newcommand{\Pq}{\mathbb{P}^{14}_k}
\newcommand{\Hilbz}{\text{Hilb}_{24} \bigl( \Pdzd \bigr)}
\newcommand{\Hilb}{\text{Hilb}_{24} \bigl( \bpd \bigr)}
\newcommand{\Sym}{\text{Sym}^{24} \bigl( \bpd \bigr)}

\newcommand{\wtF}{\widetilde\F}
\newcommand{\wtP}{\widetilde{\mathbb{P}}_k^{14}}
\newcommand{\wtD}{\widetilde{\mathscr{D}}}
\newcommand{\usm}{U^{\rm sm}_{10}}

\newcommand{\mult}{{\rm mult}}
\newcommand{\nom}{derived}

\newcommand{\bilq}[2]{\left\langle #1 , #2 \right\rangle_2}
\newcommand{\bitq}[2]{\left\langle #1 , #2 \right\rangle_3}

\newcommand{\Fr}{{\rm Fr}}

\newcommand{\qev}{q_\varepsilon}
\newcommand{\qep}{\qev (x,y,z)}
\newcommand{\car}{{\rm char}}
\newcommand{\HH}{{\rm H}^0}
\newcommand{\Vinf}{\mathscr{V}^{\text{inf}}_\ell}
\newcommand{\Vinfl}[1]{\mathscr{V}^{\text{inf}}_{\ell_#1}}
\newcommand{\VH}{\mathscr{V}^{H}_\ell}

\newcommand{\VHl}[1]{\mathscr{V}^{H}_{\ell_#1}}
\newcommand{\pd}[2]{\left[ #1 , #2 \right] }
\newcommand{\GL}{\mathfrak{gl}_3}
\newcommand{\SL}{\mathfrak{sl}_3}

\address{Marco Pacini, Instituto de Matem\'atica, Universidade Federal Fluminense, Rio de Janeiro, Brazil}
\email{pacini.uff@gmail.com, pacini@impa.br}

\address{Damiano Testa, Mathematics Institute, University of Warwick, Coventry, CV4 7AL, United Kingdom}
\email{adomani@gmail.com}

\begin{document}

\title{Reconstructing general plane quartics \\
from their inflection lines}

\author{Marco Pacini}
\author{Damiano Testa}

\begin{abstract}
Let $C$ be a general plane quartic and let $\fl(C)$ denote the configuration of inflection lines of $C$.  We show that if $D$ is any plane quartic with the same configuration of inflection lines $\fl(C)$, then the quartics $C$ and $D$ coincide.
\end{abstract}

\thanks{The first author was partially supported by CNPq, processo 200377/2015-9 and processo 301314/2016-0}
\thanks{The second author was partially supported by EPSRC grant number EP/F060661/1 and number EP/K019279/1.}

\maketitle

\tableofcontents

\section*{Introduction}

Let $k$ be an algebraically closed field and let $C \subset \bp$ be a smooth quartic curve over $k$.  An {\emph{inflection line}} for the quartic $C$ is a line $\ell$ in $\bp$ such that the intersection $\ell \cap C$ contains of a point $p$ of multiplicity at least~$3$.  If $k$ is the field of complex numbers, then the Pl\"ucker formulas imply that $C$ admits exactly $24$ inflection lines, counted with multiplicity.

The present paper is motivated by the following question.
\begin{equation} \label{eq:que}
\tag{$\mathscr{R}$}
{\emph{Does the set of inflection lines of a general plane quartic curve uniquely determine the curve?}}
\end{equation}
The answer to Question~\eqref{eq:que} may depend on the characteristic of the ground field $k$.  Our main result is an affirmative answer to Question~\eqref{eq:que}, when the characteristic of the field $k$ is coprime with~$6$.  We work as much as we can with fields of arbitrary characteristic.

\subsubsection*{Previous work}
Question~\eqref{eq:que} is inspired by a similar question on bitangent lines to plane quartics, introduced and addressed in a paper by Caporaso-Sernesi~\cite{CS1}.  A subsequent paper of Lehavi~\cite{L} extended the result to all smooth quartic curves and again work of Caporaso-Sernesi~\cite{CS2} generalized the question and resolved it for general canonical curves.

For the specific question on inflection lines, we had already obtained weaker results.  In~\cite{PT1}, we solved the analogous question for plane cubics; we were also able to reconstruct a general plane quartic from the knowledge of its inflection lines and a {\emph{simple}} inflection point.  In~\cite{PT2}, we examined some families of plane quartics found by Vermeulen~\cite{V}.  From these families, we defined {\emph{Vermeulen's list}}, a list of quartics that, over fields of characteristic $0$, consists of all the smooth plane quartics with at least $8$ hyperinflection lines.  We were able to show that different quartics in Vermeulen's list have different configurations of inflection lines (with only two exceptions).  Since we only analyzed configurations of inflection lines of smooth plane quartics, our results did not answer Question~\eqref{eq:que}.

\subsubsection*{Methods}
We set up our problem using classical invariant theory of binary and ternary quartic forms.  We thus construct equations for the configuration of inflection lines associated to general plane quartics.  This produces naturally a rational map from plane quartics to the Hilbert scheme of~$24$ points in $\bpd$.  Resolving this map, we associate configurations of inflection lines also to singular curves (Definition~\ref{def:conf}).  Our proof then is by a degeneration argument.  We fix a special smooth quartic $C$ and we analyze the quartics $D$ admitting the same configuration of inflection lines as $C$.  We follow roughly three main steps.
\begin{enumerate}
\item
\label{it:lis}
The quartic $D$ cannot be singular.
\item
\label{it:rilis}
If $D$ is smooth, then $D$ coincides with $C$.
\item
\label{it:defo}
If the special quartic $C$ can be reconstructed, then a general quartic can be reconstructed.
\end{enumerate}

In Step~\eqref{it:lis}, we look for properties of configurations of inflection lines associated to plane quartics that will distinguish between smooth and singular quartics.

We study the local structure of the classical contravariants of plane quartics.  If $D$ is a singular quartic not containing a triple line, then the numerical properties of the configurations of inflections lines are already enough for our purposes.

To analyze quartics containing a triple line, we follow a more global approach.  We combine deeper facts about the invariant theory of ternary quartic forms with Lie algebra techniques to generate restrictions satisfied by configurations of inflection lines associated to such quartics.  The resulting configurations are different from the configurations of {\emph{general}} quartics: the precise control that we obtain using this global approach provides us with an explicit condition implying this genericity assumption.  In summary, let $\fl \subset \bpd$ be a configuration of inflection lines associated to a singular plane quartic.  Thus, $\fl$ corresponds to a point in $\Hilb$.  We show that $\fl$ satisfies one of the following properties:
\begin{itemize}
\item
one point of $\fl$ has multiplicity at least $3$;
\item
the reduced subscheme of $\fl$ is contained in a singular quartic.
\end{itemize}
Neither of these properties is satisfied by the configuration of inflection lines of a general quartic.  This concludes Step~\eqref{it:lis}.

So far, the argument carries through with minor differences, whether we encode a configuration as a point in the Hilbert scheme or in the symmetric product.  The reason why we choose to work with $\Hilb$ rather than with $\Sym$, shows up in the remaining two steps.  We exploit the added information carried by the non-reduced ideals in the Hilbert scheme.  Indeed, the configuration of inflection lines of the special quartic $C$ that we use for the reconstruction is non-reduced.  This has two consequences.

On the one hand, in Step~\eqref{it:rilis}, the reconstruction process becomes easier: the problem of going from the inflection lines to the curve $C$ becomes almost linear in the non-reduced setting.  On the other hand, in Step~\eqref{it:defo}, the more non-reduced the configuration becomes, the more singular the deformation spaces become.

We need to strike a balance between allowing enough non-reducedness to successfully perform Step~\eqref{it:rilis}, while at the same time maintaining enough smoothness to deform away from the non-reduced locus and complete Step~\eqref{it:defo}.

We study carefully the deformation space in the presence of non-reduced points.  An inflection line corresponding to a reduced point in the configuration represents a smooth point of the deformation space and imposes {\emph{two}} linear conditions on the tangent space.  An inflection line corresponding to a non-reduced point, represents a singular point of the deformation space, but we are still able to identify {\emph{one}} linear condition on the tangent space.  This apparently small gain turns out to be essential for our argument.  Ultimately, the source of this improvement originates from the smoothness of the Hilbert scheme.  Amusingly, over fields of characteristic~$13$, we find two plane quartics with the following properties.  The quartics are smooth, distinct and projectively equivalent.  Their configurations of inflection lines are different points in $\Hilb$, having the same image in $\Sym$.  Our tangent space computations show that either one of these two examples can be used in the reconstruction argument.

\subsubsection*{Outline}
In Section~\ref{sec:invcl}, we define the invariants and contravariants that we will use.  This allows us to give equations for the configurations of inflection lines.  We define {\emph{totally harmonic quartics}} by the vanishing of the quartic contravariant (Definition~\ref{def:tothar}).  Lie algebras also appear in this section.  The interplay between invariant theory and Lie algebras explains our reasoning in a completely satisfactory way, if it were not for the mysterious identity of Lemma~\ref{lem:trimi}: we do not know how to justify it conceptually.

In Section~\ref{sec:alli}, we introduce and analyze the natural map $\F$ assigning to a general plane quartic its configuration of inflection lines in the Hilbert scheme of points in $\bpd$.  We classify totally harmonic quartics in all characteristics in Proposition~\ref{prop:nonsva}.  As a byproduct, we deduce that the only smooth plane quartics not having $24$ inflection lines, counted with multiplicity, are projectively equivalent to the Fermat quartic curve over algebraically closed fields of characteristic $3$ (Corollary~\ref{cor:fermat}).

In Section~\ref{sec:trop}, we argue that the configurations of inflection lines associated to singular plane quartics (see Definition~\ref{def:conf}) are disjoint from the configurations of inflection lines of general quartics (Propositions~\ref{prop:liscsing} and~\ref{prop:re3}).

In Section~\ref{sec:rifi}, we identify a smooth plane quartic curve $V$ that we can explicitly reconstruct from its configuration of inflection lines (Corollary~\ref{cor:rico}).  Combining this fact with the conclusions of Section~\ref{sec:trop}, we deduce that, for any resolution of $\F$, the fiber above the quartic $V$ consists of $V$ alone.  All that is left to show is that $V$ is reduced as a fiber of $\F$: this is Lemma~\ref{lem:et}.  In fact, we can reconstruct explicitly a positive dimensional family of non-isomorphic plane quartics with our argument.  Remark~\ref{rem:hilsym} highlights the features that the Hilbert scheme has and that the symmetric product lacks.  Finally, Theorem~\ref{thm:rf} answer Question~\eqref{eq:que}.

\subsection*{Acknowledgements}
The authors would like to thank Igor Dolgachev for useful conversations on the classical invariant theory of binary and ternary quartic forms.  We would like to thank the anonymous referee for a careful reading of the paper and constructive criticism: we feel that our paper greatly improved thanks to the comments of the referee.  Also, even though the results presented in this paper require no computer calculation, the authors spent several hours doing computations with Magma~\cite{magma}: we benefited immensely from using it.  We are happy to send by email the Magma code that we used.

\section{Classical invariant theory of quartic forms} \label{sec:invcl}

In this section, we identify the locus of binary quartic forms with a triple root.  Classically, this ties in with the invariant theory of binary quartic forms and is our starting point.  See~\cite{sal}, especially Chapter~VI, for more details and complements.  We then link invariant theory with the basic theory of Lie algebras to obtain several identities that we use in the remainder of the paper.

Let $R$ be a commutative ring with identity and let $m,n$ be non-negative integers.  We denote by $R[x_1,\ldots,x_n]$ the polynomial ring over $R$ in $n$ variables $x_1 , \ldots , x_n$ and by $R[x_1,\ldots,x_n]_m$ the submodule of forms of degree $m$.  Let $\mathbb{P}^n_R$ denote the projective space of dimension $n$ over the spectrum $\spR$ of $R$.  For most of our applications, the ring $R$ will be either the ring of integers $\mathbb{Z}$ or an algebraically closed field $k$ of arbitrary characteristic.

Let $f_0 , \ldots , f_4$ be homogeneous coordinates on $\PR$.  We identify $\PR$ with the projectivization of the space of binary quartic forms associating to the point $[f_0 , \ldots , f_4]$ in $\PR$ the binary quartic form
\begin{eqnarray*}
f = f_0 y^4 + f_1 y^3 z + f_2 y^2 z^2 + f_3 y z^3 + f_4 z^4 \in R[y,z]_4.
\end{eqnarray*}

The group-scheme $GL_{2,R}$ over $R$ acts on binary forms in $\PR$ by the rule
\[
\begin{pmatrix}
a & b \cr c & d
\end{pmatrix}
f(y,z) = f(ay+bz,cy+dz).
\]
There are two basic invariants under the action of $SL_{2,R} \subset GL_{2,R}$ on binary quartic forms:
\begin{equation} \label{e:inva}
\begin{array}{rcl}
S(f) & = & 12 f_0 f_4 -3 f_1 f_3 + f_2^2 , \\[5pt]
T(f) & = & 72 f_0 f_2 f_4 - 27 f_0 f_3^2 - 27 f_1^2 f_4 + 9 f_1 f_2 f_3 - 2 f_2^3 .
\end{array}
\end{equation}
These invariants are denoted by $I,J$ in~\cite{N}*{p.~96}.  Denote by $V_3 \subset \PR$ the scheme defined by the vanishing of $S$ and $T$ and by $U_3 \subset \spR$ the open subset $U_3 = \Spec (R[\frac{1}{3}])$ of $\spR$.  Let $V_3' \subset \PR$ denote the image of the morphism
\[
\begin{array}{rcl}
\PuR \times \PuR & \longrightarrow & \PR \\[5pt]
([\alpha , \beta] , [\gamma , \delta ]) & \longmapsto & (\alpha y - \beta z)^3  (\gamma y - \delta z)
\end{array}
\]
with the reduced induced subscheme structure.  We refer informally to the scheme $V_3'$ as the locus of binary quartic forms with a triple root.

\begin{Lem}\label{lem:z3}
The morphism $V_3 \to \spR$ is flat over the open set $U_3 \subset \spR$.  Moreover, the schemes $V_3$ and $V_3'$ coincide over the same open set $U_3$.
\end{Lem}

\begin{proof}
It suffices to show the lemma in the case in which $R=\mathbb{Z}$, since flatness is stable under base-change.

To prove flatness, we will show that the scheme $V_3$ is the complete intersection of $S$ and $T$ of codimension~$2$ above all primes of $U_3$.  This suffices by~\cite{sp}*{\href{http://stacks.math.columbia.edu/tag/00R4}{Tag 00R4}}, since $\spz$ is regular and $V_3$ is Cohen-Macaulay as a consequence of being a complete intersection on $U_3$.

The quadric $V(S)$ defined by the vanishing of $S$ has rank $3$ at the prime $(2)$, rank $1$ at the prime $(3)$ and rank $5$ at all remaining primes.  Therefore $V(S)$ is irreducible at every prime (and it is non-reduced at $(3)$).  Thus, to prove flatness at a prime $(p)$, it suffices to show that there are binary quartic forms contained in $V(S)$ on which $T$ does not vanish.  The form $S(y,z) = y^4 - y z^3$ is one such example at all primes $p \neq 3$, since $S(h) = 0$ and $T(h) = - 27$.  This proves the first part of the statement.

We now prove the second part of the statement.  First, it is an immediate check that the forms $S$ and $T$ vanish identically on $V_3'$, so that the locus $V_3'$ is contained in $V_3$.  Next, observe that $V_3$ has pure codimension $2$, just like $V_3'$, so that the intersection of $V_3$ with any $2$-dimensional subvariety of $\Pz$ contains at least one point from each irreducible component of $V_3$.  We are going to show that the plane $\pi$ with equations $f_0=f_4=0$ intersects $V_3$ at {\emph{smooth}} points corresponding to quartic forms in $V_3'$.  It will then follow that every component of $V_3$ is also a component of $V_3'$ and we will be done.

The intersection of $V_3$ and $\pi$ consists of the forms with coefficients satisfying the system
\[
\begin{array}{rcl@{\hspace{20pt}}rcl}
f_0 & = & 0; & -3 f_1 f_3 + f_2^2 & = & 0; \\[5pt]
 f_4 & = & 0; & f_2 (9 f_1 f_3 - 2 f_2^2) & = & 0.
\end{array}
\]
Thus, away from the prime $(3)$, the intersection $V_3 \cap \pi$ consists of the closed points corresponding to the forms $y^3z$ and $yz^3$ that are clearly in $V_3'$.  The Zariski tangent space to $V_3$ at the point $[0,1,0,0,0]$ has equations $-3f_3 = -27 f_4 = 0$, so that $[0,1,0,0,0]$ is smooth on $V_3$ at every prime of $U_3$.  Exchanging the roles of $y$ and $z$, the same is true for the point $[0,0,0,1,0]$.
\end{proof}

\begin{Rem}
If $R$ is the field $\mathbb{Z}/3\mathbb{Z}$, the invariants $S$ and $T$ become $f_2^2$ and $f_2^3$, so that they no longer define the locus of binary quartic forms with a triple root.  It is an easy computation to show that the flat limit of $V_3 \to \Spec(\mathbb{Z})$ at $(3)$ is the scheme defined by the ideal $(f_2 , f_0 f_4 - f_1 f_3)^2$.  The scheme defined by the radical ideal $(3 , f_2 , f_0 f_4 - f_1 f_3)$ is the reduced subscheme of $V_3$ consisting of polynomials over a field of characteristic~$3$ with a root of multiplicity at least~$3$: it is a smooth quadric of dimension $2$.  We observe that adding to the classical invariants the following two expressions
\begin{equation} \label{e:clch3}
\begin{array}{lccl}
S_3(f) = & \frac{1}{3} \left( T(f) + 2 f_2 S(f) \right) & = &
32 f_0 f_2 f_4 - 9 f_0 f_3^2 - 9 f_1^2 f_4 + f_1 f_2 f_3
\\[5pt]
S_4(f) = & \frac{1}{3} (S_3(f) f_2 + S(f) (f_0 f_4 - f_1 f_3)) & = & -128 f_0^2 f_4^2 + 28 f_0 f_1 f_3 f_4 - 3 f_0 f_2 f_3^2 - 3 f_1^2 f_2 f_4 + f_1^2 f_3^2
\end{array}
\end{equation}
we obtain a subscheme of $V_3$ whose structure map to $\spz$ is flat also above $(3)$.  For our main result, we exclude the case of fields of characteristic $3$.
\end{Rem}

The invariant $S$ is a quadratic form on $\PR$ and we will use the associated bilinear form in our arguments.  We denote this bilinear form by $\bilq{-}{-}$: if $f = \sum f_i y^{4-i} z^i$ and $g = \sum g_i y^{4-i} z^i$ are forms in $R[y,z]_4$, then we have
\[
\bilq{f}{g} = 12 f_0 g_4 - 3 f_1 g_3 + 2 f_2 g_2 - 3 f_3 g_1 + 12 f_4 g_0 .
\]
We will also use the identity
\[
S (f+g) = S (f) + \bilq{f}{g} + S (g).
\]

\begin{Rem} \label{rem:orbite}
Let $R=k$ be an algebraically closed field.  The vanishing set of $S$ in $\mathbb{P}^4_k$ decomposes into three $SL_2(k)$-orbits:
\begin{itemize}
    \item one orbit of forms with no repeated roots;
    \item the orbit of forms with a root of multiplicity exactly $3$;
    \item the orbit of forms with a root of multiplicity exactly $4$.
\end{itemize}
Indeed, in each $SL_2(k)$-orbit of forms having at least three distinct roots, there is a representative proportional to $yz(y-z) (y - \lambda z)$, for some $\lambda \in k$.  We then conclude easily evaluating the invariant $S$ on such a representative.  If a form has at most two distinct roots and no root of multiplicity at least $3$, then it is equivalent to $\lambda y^2z^2$, for some non-zero $\lambda \in k$, and $S( \lambda y^2z^2)= \lambda^2 $ does not vanish.  We already saw that $S$ vanishes on forms with roots of multiplicity at least $3$.
If the characteristic of $k$ is different from $3$, the form $y^4 - yz^3$ has distinct roots and $S(y^4-yz^3)$ vanishes.  If the characteristic of $k$ is $3$, then the form $y^4+z^4$ has distinct roots and $S(y^4+z^4)$ vanishes.
\end{Rem}

We now move on to homogeneous polynomials $q(x,y,z) \in R[x,y,z]_4$ of degree $4$ in three variables $x,y,z$.  Let $\PdR$ be the projective plane with homogeneous coordinates $x,y,z$.  We denote by $\PdRd$ the dual projective plane with coordinates $u,v,w$ dual to the coordinates $x,y,z$.  Let $\PqR$ denote the projective space of dimension $14$ that we think of as the space of quartic curves in $\PdR$.  We extend the definition of the invariants $S$ and $T$ to contravariant forms $H$ and $K$ on the space $\PqR$:
\[
\begin{array}{rcl}
    H(q)(u,v,w) & = & u^4 S \left( q \left( - \frac{v}{u} y - \frac{w}{u} z , y , z \right) \right) , \\[5pt]
    K(q)(u,v,w) & = & u^6 T \left( q \left( - \frac{v}{u} y - \frac{w}{u} z , y , z \right) \right) .
\end{array}
\]
The expressions above are clearly rational sections of $\mathscr{O}_{{\PdR}^\vee}(4)$ and  of $\mathscr{O}_{{\PdR}^\vee}(6)$; an easy check using the invariance of $S$ and $T$ under $SL_{2,R}$ shows that they are in fact global sections of the corresponding sheaves: they are ternary forms of degrees $4$ and $6$ respectively.

\begin{Def} \label{def:harmo}
Let $C \subset \PdR$ be the plane quartic curve given by the vanishing of a ternary quartic form $q(x,y,z) \in R[x,y,z]_4$.  We call the ternary quartic form $H(q)(u,v,w) \in R[u,v,w]_4$ the {\emph{harmonic quartic associated to $q$}}; we call the ternary sextic form $K(q)(u,v,w) \in R[u,v,w]_6$ the {\emph{harmonic sextic associated to~$q$}}.  We denote by $H(C)$ and $K(C)$ in $\PdRd$ the schemes defined by the vanishing of $H(q)$ and $K(q)$ respectively.
\end{Def}

\begin{Def}\label{def:tothar}
Let $q(x,y,z) \in R[x,y,z]_4$ be a ternary quartic; we say that $q$ is {\emph{totally harmonic}} if the harmonic quartic $H(q)$ vanishes.  Similarly, if $C \subset \PdR$ is a plane quartic curve, we say that $C$ is totally harmonic if a ternary quartic form defining $C$ is totally harmonic.
\end{Def}

For instance, if $x^3$ divides $q(x,y,z) \in R[x,y,z]$, then $q(x,y,z)$ is totally harmonic.

The coefficients of $H(q)$ and $K(q)$ are forms of degree $2$ and $3$ respectively on the space $\PqR$ of quartic curves.  We shall make repeated use of the bilinear map associated to $H$: if $q,r$ are ternary quartic forms in $R[x,y,z]_4$, we define
\begin{equation} \label{e:bil3}
\bitq{q}{r} = u^4 \bilq{ q \left( - \frac{v}{u} y - \frac{w}{u} z , y , z \right) }{ r \left( - \frac{v}{u} y - \frac{w}{u} z , y , z \right) }
\end{equation}
and again observe that the equality
\begin{equation} \label{e:bil}
H(q+r) = H(q) + \bitq{q}{r} + H(r)
\end{equation}
holds.

Let $\ell \subset \PdR$ be the line with equation
\[
\ell \colon u_0 x + v_0 y + w_0 z = 0
\]
and fix a parameterization of $\ell$.  Let $q(x,y,z)$ and $r(x,y,z)$ be ternary quartic forms in $R[x,y,z]_4$ and denote by $q|_\ell$ and $r|_\ell$ the restrictions of $q$ and $r$ to the line $\ell$ using the chosen parameterization: these restrictions are therefore binary quartic forms.  It follows at once from the definitions that there is a unit $\lambda \in R^\times$ such that the identities
\begin{equation} \label{e:eval}
\left\{
\begin{array}{lcl}
H(q)(u_0 , v_0 , w_0) & = & \lambda ^4 S(q|_\ell) \\[5pt]
\bitq{q}{r}(u_0 , v_0 , w_0) & = & \lambda ^4 \bilq{q|_\ell}{r|_\ell} \\[5pt]
K(q)(u_0 , v_0 , w_0) & = & \lambda ^6 T(q|_\ell) \\[5pt]
\end{array}
\right.
\end{equation}
hold.  In what follows, we shall be concerned mostly with the vanishing of these expressions, so that the constant $\lambda $ is harmless.

\begin{Def} \label{def:infsc}
Let $C \subset \PdR$ be a plane quartic and let $q \in R[x,y,z]_4$ be a ternary quartic form defining $C$.  The {\emph{inflection scheme of $C$}} is the subscheme of $\PdRd$ defined by the simultaneous vanishing of the forms $H(q) , K(q) \in R[u,v,w]$; we denote the inflection scheme of $C$ by $\fl(C)$.  An {\emph{inflection line of $C$}} is a line corresponding to a point in the inflection scheme of $C$.
\end{Def}

\begin{Rem} \label{rem:ptitripli}
Let $R=k$ be a field of characteristic different from $3$ and let $q$ be a ternary quartic form in $k[x,y,z]_4$.  As we saw in Equation~\eqref{e:eval}, the forms $H(q)$ and $K(q)$ both vanish on the coordinates of a line $\ell$ if and only if the invariants $S(q|_\ell)$ and $T(q|_\ell)$ also vanish.  It follows from Lemma~\ref{lem:z3} that the restriction $q|_\ell$ is a binary quartic form with a root of multiplicity at least $3$.  Thus, the inflection scheme of a plane quartic $C \subset \bp$ consists exactly of the lines having a point of intersection multiplicity at least~$3$ with $C$.  In particular, for smooth curves, the points of the inflection scheme correspond to some tangent lines to the curve.  Using the dimension of the inflection scheme, we obtain a stratification of the space of plane quartics.
\end{Rem}

\begin{Rem}
Let $R=k$ be a field of characteristic $3$.  In this case, it would make sense to extend the equations in~\eqref{e:clch3} to ternary forms and use these to define the inflection scheme.  If $C$ is a smooth plane quartic over $k$, then the scheme that we are calling the inflection scheme of $C$ contains all inflection lines of $C$, as well as lines that are not inflection lines.  Indeed, the dimension of the inflection scheme that we just defined is always at least~$1$.  Since we will not use inflection schemes over fields of characteristic~$3$, we will not pursue this here.
\end{Rem}

We list in Table~\ref{tab:contih} some computations involving the harmonic quartic $H(-)$ and the associated bilinear form $\bitq{-}{-}$; we will freely use these equations that are simple consequences of the definitions.  Recall that the identity~\eqref{e:bil} holds.  Let $q_4(y,z)$, $q_3(y,z)$, $q_2(y,z)$, $q_1(y,z) \in R[y,z]$ be binary forms of degree $4,3,2,1$ respectively in $y,z$.

\begin{table}[h]
\[
\begin{array}{|rcl|}
\hline
\bitq{x^a q_{4-a}(y,z)}{x^b q_{4-b}(y,z)} & = & 0 , \textrm{ if } a+b \ge 5 \\[4pt]
\bitq{x^4}{q_4(y,z)} & = & 12 q_4 (-w,v) \\[4pt]
\bitq{x^3 q_1(y,z)}{x q_3(y,z)} & = & -3 q_1(-w,v) q_3(-w,v) \\[4pt]
\bitq{x^3 y}{q_4(y,z)} & = & - 3 u \partial_v q_4(-w,v) \\[4pt]
H(x^2 q_2(y,z)) & = & q_2(-w,v)^2 \\[4pt]
\bitq{x^2 y^2}{q_4(y,z)} & = & u^2 \partial_v \partial_v q_4(-w,v)
\\[3pt]
\hline \end{array}
\]
\caption{Some identities involving $\bitq{-}{-}$ and $H(-)$}
    \label{tab:contih}
\end{table}

\begin{Exa} \label{exa:klein}
Let $R = \mathbb{Z}$ be the ring of integers and let $C \subset \Pdz$ denote the plane Klein quartic curve with equation $q(x,y,z) = x^3 y + y^3 z + z^3 x = 0$.  In this case, the contravariants $H$ and $K$ evaluate to
\[
\begin{array}{rcl}
    H(q)(u,v,w) & = & 3 (u^3 v + v^3 w + w^3 u ) , \\[5pt]
K(q)(u,v,w) & = & 27 (u^5 w + v^5 u + w^5 v - 5 u^2 v^2 w^2) .
\end{array}
\]
Reducing these equations modulo primes $p$ different from $3$, it follows at once that the inflection scheme of $C$ is the complete intersection of $H(C)$ and $K(C)$ of dimension~$0$ and degree~$24$.  Thus, the same is true for a general plane quartic over any field of characteristic different from $3$.
\end{Exa}

\begin{Rem}
If in the ring $R$ the identity $3=0$ holds, then the harmonic quartic $H(-)$ is always a square: this is a consequence of the fact that, in this case, the invariant $S$ is the square of $f_2$, in the notation of Equation~\eqref{e:inva}.  Explicitly, let $q(x,y,z)$ be a quartic form in $R[x,y,z]_4$ and write $q(x,y,z) = q_{1}(x,y,z) + q_{2}(x,y,z) + q_{3}(x,y,z)$, where
\begin{itemize}
    \item $q_{1}(x,y,z)$ is the sum of all the terms of $q$ corresponding to the monomials $x^2 y z$, $x y^2 z$, $x y z^2$;
    \item $q_{2}(x,y,z)$ is the sum of all the terms of $q$ corresponding to the monomials $x^2 y^2$, $x^2 z^2$, $y^2 z^2$;
    \item $q_{3}(x,y,z) = q(x,y,z)-q_{1}(x,y,z) - q_{2}(x,y,z)$.
\end{itemize}
The formula for $H(q)$ becomes
\begin{equation} \label{e:quac3}
H(q) = \left( \frac{q_{2}(vw,uw,uv) - q_{1}(vw,uw,uv)}{(uvw)^2} \right) ^2.
\end{equation}
\end{Rem}

In the remainder of the section, we prove some identities relating the coefficients of a ternary quartic form and of its associated harmonic quartic.

We define an $R$-bilinear pairing $\pd{-}{-}$ on $R[x,y,z]_4 \times R[u,v,w]_4$.  As $i,j,k$ ranges among exponents of the monomials in $R[x,y,z]_4$, the rational numbers $\frac{i!j!k!}{2}$ are all integers and we define $\pd{-}{-}$ on monomials as follows:
\[
\hspace{20pt}
\begin{array}{rcl}
\pd{-}{-} \colon
R[x,y,z]_4 \times R[u,v,w]_4
& \longrightarrow &
\quad R \\
(x^i y^j z^k , u^l v^m w^n)
& \longmapsto &
\pd{x^i y^j z^k}{u^l v^m w^n} =
\left\{
\begin{array}{cl}
\displaystyle
\frac{i!j!k!}{2}
& {\textrm{if }}(i,j,k) = (l,m,n), \\[7pt]
0 & {\textrm{if }}(i,j,k) \neq (l,m,n).
\end{array}
\right.
\end{array}
\]

Let $q(x,y,z) = \sum a_{ijk} x^i y^j z^k$ be a ternary quartic form in $R[x,y,z]_4$ and let $H(q)(u,v,w) = \sum h_{ijk} u^i v^j w^k$ be the associated harmonic quartic form in $R[u,v,w]_4$.  Recall that the coefficients $h_{ijk}$ of $H(q)$ are quadratic forms in the coefficients $a_{ijk}$ of $q$.  In our reasoning, we shall use the trilinear form $t(-,-,-)$ defined by
\begin{eqnarray*}
t \colon R[x,y,z]_4 \times R[x,y,z]_4 \times R[x,y,z]_4
& \longrightarrow &
R \\[5pt]
(q_1,q_2,q_3)
& \longmapsto &
t(q_1,q_2,q_3) = \pd{q_1}{\bitq{q_2}{q_3}}.
\end{eqnarray*}
In order to establish the identities that we use, we exploit the following lemma.

\begin{Lem} \label{lem:trimi}
The trilinear form $t$ is symmetric in its arguments, that is, for all forms $q_1,q_2,q_3 \in R[x,y,z]_4$ and all permutations $\sigma \in \mathfrak{S}_3$, the identity
\[
t(q_1,q_2,q_3) =
t(q_{\sigma(1)},q_{\sigma(2)},q_{\sigma(3)})
\]
holds.
\end{Lem}

\begin{proof}
Since the bilinear form $\bitq{-}{-}$ is symmetric, the trilinear form $t$ is symmetric under the transposition of its second and third argument: $t(q_1,q_2,q_3)$ and $t(q_1,q_3,q_2)$ are equal.  Thus, to conclude it suffices to show that the identity $t(q_1,q_2,q_3) = t(q_2,q_1,q_3)$ holds.  This is the result of a standard calculation, that we omit.
\end{proof}

The expression
\begin{equation} \label{eq:Ainv}
A(q) = t(q,q,q) = \pd{q}{H(q)} = \sum_{i,j,k \geq 0} \frac{i!j!k!}{2} a_{ijk} h_{ijk}
\end{equation}
is an invariant of ternary quartic forms under $SL_{3,R}$: see~\cite{sal}*{p.~251-252, \S292-293}.

We denote by $\GL(R)$ the Lie algebra of $GL_3(R)$, that is, the free $R$-module of rank $9$ generated by the symbols $\xi \partial_\eta $ as $\xi$ and $\eta$ range among the variables $x,y,z$.  We define $\SL(R) \subset \GL(R)$ as the submodule where the sum of the coefficients of $x\partial_x , y\partial_y , z\partial_z$ vanishes: this is the Lie algebra of $SL_3(R)$.

The Lie algebra $\GL(R)$ acts on $R[x,y,z]$ as differential operators on polynomials.  By construction, this action preserves the $R$-module of ternary forms of any degree.  We obtain an action of $\GL(R)$ on $R[x,y,z]_4$.

\begin{Prop} \label{prop:lie}
For every element $\mathfrak{g}$ in $\SL(R)$, the identity
\[
t (\mathfrak{g} q,q,q) = \pd{\mathfrak{g} q}{ H(q) } = 0
\]
holds.
\end{Prop}

\begin{proof}
Choosing a basis for $\SL(R)$ over $R$, we can prove the statement by performing $8$ direct computations, substituting the explicit expression for the coefficients $h_{ijk}$ of $H(q)$.

Alternatively, we can argue as follows.  It suffices to show the result in the case in which $R$ is $\mathbb{Z}[a_{ijk}]$, with $a_{ijk}$ independent variables, and $q(x,y,z)$ is the quartic form $q(x,y,z) = \sum a_{ijk} x^i y^j z^k$.  Since $A(q)$ is invariant under the action of $SL_3(\mathbb{Z})$, the identity $\mathfrak{g} A= 0$ holds for every element $\mathfrak{g}$ in the Lie algebra of $SL_3(\mathbb{Z})$.  Since $t$ is trilinear and symmetric (Lemma~\ref{lem:trimi}), the identity $\mathfrak{g} t(q,q,q) = 3 t(\mathfrak{g} q,q,q)$ holds.  Combining these identities and dividing by $3$, we deduce the equality
\[
0 = t (\mathfrak{g} q,q,q) = \pd{\mathfrak{g} q}{ H(q) }
\]
for every element $\mathfrak{g} \in \SL(R)$.
\end{proof}

\begin{Cor} \label{cor:lie}
The identities
\[
\begin{array}{ll}
1. &
\displaystyle
\sum_{i,j,k \geq 0} (i-j) \frac{i!j!k!}{2} a_{ijk} h_{ijk} = 0
\\[13pt]
2. &
\displaystyle
\sum_{i,j,k \geq 0} (i+1) \frac{i!j!k!}{2} a_{(i+1)(j-1)k} h_{ijk} = 0
\\[13pt]
3. &
\displaystyle
\sum_{i,j,k \geq 0} (i+1) \frac{i!j!k!}{2} a_{(i+1)j(k-1)} h_{ijk} = 0
\end{array}
\qquad
\begin{array}{ll}
4. &
\displaystyle
\sum_{i,j,k \geq 0} (j-k) \frac{i!j!k!}{2} a_{ijk} h_{ijk} = 0
\\[13pt]
5. &
\displaystyle
\sum_{i,j,k \geq 0} (j+1) \frac{i!j!k!}{2} a_{i(j+1)(k-1)} h_{ijk} = 0
\\[13pt]
6. &
\displaystyle
\sum_{i,j,k \geq 0} (j+1) \frac{i!j!k!}{2} a_{(i-1)(j+1)k} h_{ijk} = 0
\end{array}
\]
hold.
\end{Cor}

\begin{proof}
Apply Proposition \ref{prop:lie} where $\mathfrak{g}$ in $\SL(R)$ is one of the differental operators in the following list
\[
\begin{array}{lrcl}
1. &
(x \partial_x - y \partial_y) \\[3pt]
2. &
y \partial_x \\[3pt]
3. &
z \partial_x
\end{array}
\qquad
\qquad
\begin{array}{lrcl}
4. &
(y \partial_y - z \partial_z) \hphantom{.} \\[3pt]
5. &
z \partial_y \hphantom{.}\\[3pt]
6. &
x \partial_y
\end{array}
\]
to prove the result.
\end{proof}

\section{Inflection lines of plane quartics and Hilbert schemes} \label{sec:alli}

We denote by $\Hilbz$ the Hilbert scheme over the integers parameterizing closed subschemes of dimension~$0$ and degree $24$ of $\Pdzd$.  We construct a rational map
\[
\begin{array}{rcl}
\F_{\mathbb{Z}} \colon \Pqz
& \dashrightarrow &
\Hilbz \\[3pt]
C & \longmapsto &
[\fl(C)]
\end{array}
\]
assigning to a general plane quartic $C$ the point $[\fl(C)]$ of $\Hilbz$ corresponding to the inflection scheme of $C$.  Recall that $\fl(C)$ is the intersection of the harmonic quartic $H(C)$ and sextic $K(C)$ associated to $C$.  From now on, we abuse the notation, and denote by $\fl(C)$ the point of the Hilbert scheme corresponding to the inflection scheme of $C$.  By Example~\ref{exa:klein}, the rational map $\F_{\mathbb{Z}}$ is defined on an open set whose image in $\Spec(\mathbb{Z})$ contains $\Spec(\mathbb{Z}[\frac{1}{3}])$.  By the valuative criterion of properness, the rational map $\F_{\mathbb{Z}}$ extends also to points lying above the prime $(3)$ of $\Spec(\mathbb{Z})$.

For the remainder of the section, we let $k$ denote a field.  The map $\F_{\mathbb{Z}}$ restricts a rational map
\[
\F \colon \Pq \dashrightarrow \Hilb .
\]
Let $\wtP$ be the closure in $\Pq\times \Hilb$ of the graph
\[
\bigl\{(C, \fl(C)) : C \text{ is a general plane quartic} \bigr\}
\]
and let
\[
\wtF \colon \wtP \longrightarrow \Hilb
\]
be the restriction to $\wtP$ of the projection of $\Pq\times \Hilb$ onto the second factor.

The problem of reconstructing a general plane quartic from its configuration of inflection lines is the question of deciding if the morphism $\wtF$ is birational onto its image.

\begin{Def} \label{def:rec}
Let $C \subset \bp$ be a plane quartic.  We say that $C$ is {\emph{reconstructible from a point $\fl \in \Hilb$}} if the fiber $\wtF^{-1}(\fl)$ only contains the point $(C,\fl)$.
\end{Def}

The existence of a quartic curve that is reconstructible from a point $\fl$ is not enough to conclude that the map $\wtF$ is birational onto its image: this simply implies that the fiber over $\fl$ consists of a single point, but not that it is reduced.

\begin{Def} \label{def:conf}
If $C$ is a plane quartic and $\fl$ is a point in $\Hilb$, we say that $\fl$ is a {\emph{configuration of inflection lines associated to $C$}} if the pair $(C , \fl )$ is in $\wtP$.
\end{Def}

Thus, the curve $C$ is reconstructible from $\fl$ if $\fl$ is a configuration of inflection lines associated to $C$ and no other plane quartic admits $\fl$ as an associated configuration of inflection lines.

There is a weaker notion of reconstructibility, where we require two plane quartics with the same configuration of inflection lines to be projectively equivalent and not necessarily equal.  We will work with the stricter notion (Definition~\ref{def:rec}), although the weaker one is well-suited for the moduli space of curves of genus~$3$.

We now study the indeterminacy locus of the rational map $\F$ and we give an explicit characterization of the totally harmonic quartic curves.  The totally harmonic quartic curves lie in the indeterminacy locus of $\F$ and play a crucial role in our argument.

\begin{Lem} \label{lem:defU}
Let $k$ be a field of characteristic coprime with~$6$ and let $C$ be a reduced quartic whose singular points have multiplicity~$2$.  The inflection scheme $\fl(C) \subset \bpd$ of $C$ is the complete intersection of $H(C)$ and $K(C)$.  In particular, $\fl(C)$ is a subscheme of dimension~$0$ and length~$24$ of $\bpd$ and the rational map $\F$ is defined at $C$.
\end{Lem}

\begin{proof}
It suffices to show that if $C$ is reduced with singular points of multiplicity~$2$, then the scheme $\fl(C)$ is finite.  The lines corresponding to points in $\fl(C)$ are lines with a point of intersection multiplicity at least $3$ with $C$.  Let $\ell$ be one such line in $\fl(C)$.

Suppose first that $\ell$ contains a singular point $p$ of $C$.  In this case, $\ell$ must be in the tangent cone to $C$ at $p$ and there are only a finite number of such lines.  This in particular takes care of the case in which $\ell$ is a component of $C$.

Suppose now that $\ell$ and $C$ meet entirely at smooth points.  Let $p$ be the point of multiplicity at least $3$ in $\ell \cap C$ and let $C_p \subset C$ denote the irreducible component of $C$ containing $p$.  The curve $C_p$ is not a line, by the first part of the argument.  Therefore, the line $\ell$ corresponds to a singular point of the curve $C_p^\vee $, dual to $C_p$.  We deduce that the Gauss map associated to $C_p$ is not constant and that $\ell$ is in its branch locus.  Since the characteristic of the ground-field is coprime with~$6$, the Gauss map is separable: the line $\ell$ is contained in the finite set of ramification points of the Gauss map and we are done.
\end{proof}

The proof of Lemma~\ref{lem:defU} reduces to the one of~\cite{PT1}*{Proposition~2.5} in the case of fields of characteristic~0.  Let $C$ be a reduced plane quartic curve with singular points of multiplicity~$2$.  There is a well-defined configuration of inflection lines $\F(C)$ of degree $24$ in $\bpd$ and it coincides with $\fl(C)$.  This is not true if $C$ is non-reduced or has singular points of multiplicity at least~$3$: in these cases $\fl(C)$ is not finite and the following remark shows that there is more than one configuration of inflection lines associated to $C$.

\begin{Rem} \label{rem:flf}
Let $C \subset \bp$ be a plane quartic curve.  A line $\ell$ has intersection multiplicity at least $3$ at some point of $C$ (or is contained in $C$) if and only if $\ell$ appears in some configuration of inflection lines associated to $C$.  Indeed, if $\F$ is defined at the curve $C$, then $\F(C)$ consists of lines with intersection multiplicity at least $3$ at some point of $C$.  Since $\wtP$ is the closure of the graph of $\F$ and intersection multiplicities are upper semi-continuous, the same holds true for the lines appearing in every configuration of inflection lines associated to any plane quartic $C$, no matter how singular $C$ is.  Thus, every line appearing in every configuration of inflection lines associated to $C$ has intersection multiplicity at least $3$ with $C$.  Conversely, let $\ell$ be a line with intersection multiplicity at least $3$ with $C$ at a point $p$ (possibly, $C$ contains $\ell$).  Let $C_{\ell,p}$ be a plane quartic curve where $\F$ is defined and admitting $\ell$ as an inflection line at the point $p$.  The rational map $\F$ is defined at the generic point of the pencil generated by $C$ and $C_{\ell,p}$ and the line $\ell$ appears in the configuration associated to the generic element of the pencil.  Therefore, the line $\ell$ appears in some configuration of inflection lines associated to $C$.

If the characteristic of the field $k$ is different from $3$, by Remark~\ref{rem:ptitripli}, a line $\ell$ is in the inflection scheme $\fl(C)$ if and only if $\ell$ is in some configuration of inflection lines associated to $C$.  We list the dimensions of inflection schemes of all plane quartics in Table~\ref{tab:sindim}.
\begin{table}[h]
    \centering
    \begin{tabular}{|c|c|}
    \hline
        Singularities of $C$ & $\dim \fl(C)$ \\[3pt]
    \hline
        isolated double points & 0 \\[3pt]
        double conics &  $\geq 1$ \\[3pt]
        triple points & $\geq 1$ \\[3pt]
        isolated triple locus & $\leq 1$ \\[3pt]
        line with multiplicity at least $3$ & 2 \\[3pt]
    \hline
    \end{tabular}
    \caption{{\protect\vphantom{${W^W}^W$}}Dimensions of the inflection scheme of a plane quartic $C$}
    \label{tab:sindim}
\end{table}
\end{Rem}

In our next result, we classify totally harmonic ternary quartic forms (Definition~\ref{def:tothar}) over algebraically closed fields of any characteristic.

\begin{Prop} \label{prop:nonsva}
Let $k$ be an algebraically closed field of characteristic $p \geq 0$ and let $q(x,y,z) \in k[x,y,z]_4$ be a ternary quartic form with coefficients in $k$.  The form $q$ is totally harmonic if and only if after a linear change in the variables in $x,y,z$, the pair $(p , q)$ is equal to one of the following:
\begin{itemize}
\item
    $(p , f(y,z))$, where $f$ is one of the binary forms mentioned in Remark~\ref{rem:orbite};
\item
    $(p , x^4 + y^3 z)$, where $p \in \{ 2,3 \}$ (singularity of type $E_6$);
\item
    $(3 , x (x^2 y + z^3))$ (singularity of type $E_7$);
\item
    $(3 , x^4+y^4+z^4)$ (Fermat curve or Klein curve, they are isomorphic over algebraically closed fields of characteristic $3$).
\end{itemize}
\end{Prop}

\begin{proof}
It is an immediate check to verify that the stated forms are totally harmonic.  To prove the converse, we can clearly reduce to the case in which $q$ is not the zero form.

If the plane quartic curve $Q$ with equation $q=0$ does not have a smooth point, then it follows that $q(x,y,z)$ is the square of a not necessarily irreducible quadratic form $c(x,y,z)$.  If the quadratic form $c$ defines a conic with a smooth point, then the restriction of the quartic polynomial $q=c^2$ to a general line in $\bp$ is a polynomial with two distinct double roots.  In particular, the invariant $S$ of such a binary form is non-zero, and we conclude that $q$ is not totally harmonic.  Otherwise, the conic $c=0$ has no smooth points and hence its equation is the square of a linear form $\ell$, and clearly $q=\ell^4$ is totally harmonic and of the required form.

We have therefore reduced to the case in which the quartic curve $Q$ has smooth points: choose coordinates in $\bp$ so that the point $[1,0,0]$ lies in the smooth locus of $Q$ and the tangent line to $Q$ at $[1,0,0]$ is the line with equation $y=0$.  An equation of the curve $Q$ has the form
\begin{equation} \label{e:qforma}
x^3 y + x^2 q'_2(y,z) + x q'_3(y,z) + q'_4(y,z) = 0,
\end{equation}
where $q'_2,q'_3,q'_4$ are binary forms of degrees $2,3,4$, respectively.

Suppose first that the characteristic of the field is different from $3$.  Write $q'_2(y,z) = y q_1(y,z) + \alpha z^2$, where $q_1$ is a binary linear form and $\alpha \in k$ is a constant.  Changing coordinates by $(x,y,z) \mapsto (x-\frac{1}{3} q_1 , y , z)$, the equation for $Q$ simplifies to
\[
x^3 y + \alpha x^2 z^2 + x q_3(y,z) + q_4(y,z) = 0,
\]
for some binary forms $q_3,q_4$ of respective degrees $3,4$.  Using the equation
\[
H(x^3 y + \alpha x^2 z^2 + x q_3(y,z) + q_4(y,z)) = 0,
\]
and Equation~\eqref{e:bil}, we collect together the expressions with respect to the power of $u$ appearing in the monomials.
We only need the following:
\begin{eqnarray*}
u^0 \colon & \bitq{x^3 y}{x q_3(y,z)} + \alpha^2 H(x^2 z^2) & = 0 \\
u^1 \colon & \bitq{x^3 y}{q_4(y,z)} + \alpha \bitq{x^2 z^2}{x q_3(y,z)} & = 0 \\
u^4 \colon & H (q_4(y,z)) = S(q_4(y,z)) u^4 & = 0.
\end{eqnarray*}
We analyze the condition on $u^0$.  The equality $\alpha = 0$ follows since $\alpha^2$ is the only coefficient of $v^4$.   Moreover, from the equality $\bitq{x^3y}{xq_3(y,z)} = 3 w q_3(-w,v) = 0$, we deduce that $q_3(y,z)=0$.

We are left with the equalities $\bitq{x^3y}{q_4(y,z)} = - 3 \partial_v q_4(-w,v) = 0$ and $S(q_4(y,z)) = 0$.  If the characteristic of the ground field $k$ is not only different from $3$, but also different from $2$, then we deduce that $q_4(y,z) = \lambda y^4$, for some constant $\lambda \in k$ and we are done, since $q=x^3y + \lambda y^4$.  If the characteristic of the ground field $k$ is equal to $2$, then we obtain that $q_4(y,z) = \lambda y^4 + \mu y^2 z^2 + \nu z^4$, for some constants $\lambda,\mu,\nu \in k$.  The condition $S(q_4) = 0$ becomes $\mu^2 = 0$, so that $q(x,y,z) = x^3 y + \lambda y^4 + \nu z^4$.  If $\nu = 0$, then $q$ is of the required form; otherwise, choose $\lambda ',\nu' \in k$ satisfying $\lambda'^4=\lambda $ and $\nu'^4=\nu $ and change coordinates by $(x,y,z) \mapsto (x,y,(z+\lambda'y)/\nu')$ to transform $q$ into $x^3y + z^4$, and we conclude after a cyclic permutation of the coordinates.

We now assume that the characteristic of the ground field is $3$.  By Equation~\eqref{e:quac3} the condition $H(q) = 0$ implies that $q$ is of the form
\[
q(x,y,z) =
a_{11} x^4 + a_{12} x^3 y + a_{13} x^3 z +
a_{21} x y^3 + a_{22} y^4 + a_{23} y^3 z +
a_{31} x z^3 + a_{32} y z^3 + a_{33} z^4,
\]
for some constants $a_{ij} \in k$.  Denote by $A$ the $3 \times 3$ matrix of the coefficients of $q$:
\begin{eqnarray*}
q(x,y,z) & = &
\begin{pmatrix}
x^3 & y^3 & z^3
\end{pmatrix}
\begin{pmatrix}
a_{11} & a_{12} & a_{13} \\
a_{21} & a_{22} & a_{23} \\
a_{31} & a_{32} & a_{33}
\end{pmatrix}
\begin{pmatrix}
x \cr y \cr z
\end{pmatrix} .
\end{eqnarray*}
For any matrix $M$ with coefficients in $k$, denote by $\Fr(M)$ the matrix obtained from $M$ by applying the Frobenius automorphism to all the entries of $M$.  Explicitly, if $M$ has entries $(m_{ij})$, then $\Fr(M)$ has entries $(m_{ij}^3)$.  If $M$ is a $3 \times 3$ matrix, then the identity
\[
q \left( M \cdot
\begin{pmatrix}
x & y & z
\end{pmatrix} ^t \right)
= \begin{pmatrix}
x^3 & y^3 & z^3
\end{pmatrix}
\Fr(M)^t A M
\begin{pmatrix}
x & y & z
\end{pmatrix} ^t
\]
holds.
We are going to use the transformation rule $A \mapsto \Fr(M)^t A M$ to find a simpler form for the matrix $A$, using a suitable matrix $M$.

\subsubsection*{Case 1: the matrix $A$ is invertible}
Let $f_1,f_2,f_3$ be homogeneous coordinates on $\bp$.  First, we look for equations on the coordinates of the vector $f = (f_1 , f_2 , f_3)^t$ so that the conditions $\Fr(x,y,z) A f = 0$ and $\Fr \bigl( \Fr(f)^t A (x,y,z)^t \bigr) = 0$ coincide.
Denote by $\alpha_1, \alpha_2, \alpha_3$ the linear forms in $f$ appearing as entries of $A (f_1,f_2,f_3)^t$ and by $\beta_{1}, \beta_{2}, \beta_{3}$
the linear forms in $f_1^9,f_2^9,f_3^9$ appearing as entries of $(f_1^9,f_2^9,f_3^9) \Fr(A)$.  Define a skew-symmetric matrix $B$ and a $2 \times 3$ matrix $K$ by the formulas
\[
B = \begin{pmatrix}
0 & \beta_3 & - \beta_2 \cr
- \beta_3 & 0 & \beta_1 \cr
\beta_2 & - \beta_1 & 0
\end{pmatrix}
\quad \quad {\textrm{and}} \quad \quad
K  = \begin{pmatrix}
\alpha_{1} & \alpha_{2} & \alpha_{3} \\[4pt]
\beta_{1} & \beta_{2} & \beta_{3}
\end{pmatrix}.
\]
Imposing the proportionality of the vectors $A f$ and $\Fr(A^t) \Fr^2(f)$ is equivalent to imposing the condition that the rank of the matrix $F$ is at most $1$.  Let $\Delta_A$ be the zero locus of the three $2 \times 2$ minors of $K$
\[
\Delta_A \colon \quad
\biggl\{
(\alpha_2 \beta_3 - \alpha_3 \beta_2) = 0 , \quad \quad
(- \alpha_1 \beta_3 + \alpha_3 \beta_1) = 0 , \quad \quad
(\alpha_1 \beta_2 - \alpha_2 \beta_1) = 0 \biggr\} ,
\]
with associated Jacobian matrix
\[
J = \begin{pmatrix}
\beta_3 \nabla \alpha_2 - \beta_2 \nabla \alpha_3 \\[4pt]
- \beta_3 \nabla \alpha_1 + \beta_1 \nabla \alpha_3 \\[4pt]
\beta_2 \nabla \alpha_1 - \beta_1 \nabla \alpha_2
\end{pmatrix} =
\begin{pmatrix}
0 & \beta_3 & - \beta_2 \cr
- \beta_3 & 0 & \beta_1 \cr
\beta_2 & - \beta_1 & 0
\end{pmatrix}
A = BA
\]
with respect to the variables $f_1,f_2,f_3$.  Observe that the Jacobian matrix $J$ is the product of the skew-symmetric matrix $B$ and the matrix $A$.  Since $A$ is invertible, the rank of $B$ is $2$ and hence also the rank of $J$ is $2$ for all choices of $[f_1,f_2,f_3] \in \bp$.  It follows that the scheme $\Delta_A$ is reduced of dimension $0$ or empty.  After a general change of coordinates, we reduce to the case in which any two of the equations defining $\Delta_A$ are transverse.  In this case, we prove that there are points of $\Delta_A$ on which $q$ does not vanish.

Any two of the equations defining $\Delta_A$ imply the third, unless one among $\alpha_1,\alpha_2,\alpha_3$ vanishes.  Thus, away from the union of the three lines $\alpha_1 = 0$,  $\alpha_2 = 0$,  $\alpha_3 = 0$, any two linearly independent combinations of the equations defining $\Delta_A$ imply the third.  By our reductions, the intersection of $\Delta_A$ with any line in $\bp$ consists of a scheme of dimension $0$ and length at most $10$, since $\Delta_A$ is defined by equations of degree $10$.  It follows that $\Delta_A$ contains at least $10^2-3 \cdot 10 = 70$ points.
On the other hand, the intersection of $\Delta_A$ with the vanishing set of $q$ consists of at most $40$ points, since the degree of $q$ is $4$ and $\Delta_A$ is defined by equations of degree $10$.  We finally obtain that there are points $l$ of $\Delta_A$ satisfying $q(l) \neq 0$.  Choose vectors $l,m,n \in k^3$ as follows: $l$ lies in $\Delta_A$ and $q(l) = 1$; $m,n$ form a basis of the kernel of $\Fr(l^t) A$.  Observe that $l,m,n$ form a basis of $k^3$ and that, using coordinates $x',y',z'$ with respect to this basis, the form $q$ becomes
\[
q(x',y',z') = x'^4 + q'(y',z'),
\]
where the binary form $q'(y',z')$ satisfies $S(q') = 0$.  Using Remark~\ref{rem:orbite}, we obtain that, after a change of coordinates, the form $q$ is equal to one of the following: $x^4+y^4+z^4$, $x^4+y^3 z$, $x^4 + y^4$, $x^4$, as we wanted to show.

\subsubsection*{Case 2: the matrix $A$ is not invertible}
In this case, there is a non-zero linear combination of the rows of $A$ that vanishes: let $N$ be an invertible $3 \times 3$ matrix such that the first row of $NA$ is the zero row.  Let $M$ be the matrix $\Fr^{-1}(N^t)$ and evaluate the quartic form $q$ at $M \cdot (x,y,z)^t$, to obtain the quartic form with matrix
\[
NAM = \begin{pmatrix}
0 & 0 & 0 \cr
a_1 & a_2 & a_3 \cr
b_1 & b_2 & b_3
\end{pmatrix}.
\]
If $a_1 = b_1 = 0$, then we obtained a binary form in $y,z$ and we conclude using Remark~\ref{rem:orbite}.  Suppose therefore that $a_1,b_1$ are not both zero and, exchanging if necessary the last two rows and the last two columns of this matrix, we reduce to the case in which $b_1$ is non-zero.  After the substitution
\[
(x,y,z) \mapsto
\left( x , y , z-\Fr^{-1} \left( \frac{a_1}{b_1} \right) y \right) ,
\]
we reduce to the case in which $a_1$ vanishes.

Suppose that $a_2$ does not vanish.  After the substitution
\[
(x,y,z) \mapsto \left( x , y - \frac{a_3}{a_2} z , z \right) ,
\]
we reduce to the case in which $a_3$ vanishes.  Finally, after the substitution
\begin{equation} \label{e:sosti}
(x,y,z) \mapsto \left( x-\frac{b_2}{b_1} y-\frac{b_3}{b_1} z , y , z \right) ,
\end{equation}
we reduce to the case in which $b_2$ and $b_3$ also vanish.  After these reductions, we are left with the quartic form $b_1 x z^3 + a_2 y^4$; rescaling and permuting the variables, we obtain the form $x^4 + y^3 z$, as required.

We are still left with the case in which $a_1 = a_2 = 0$ and $b_1 \neq 0$.  Therefore, the quartic $q$ takes the form $z (a_3 y^3 + b_1 x z^2 + b_2 y z^2 + b_3 z^3)$.
Repeating the substitution in~\eqref{e:sosti}, we obtain the quartic form $z (a_3 y^3 + b_1 x z^2)$; rescaling and permuting the variables, we obtain the form $x (x^2 y + z^3)$, as required.
\end{proof}

An immediate consequence of the classification of totally harmonic quartic forms is the following corollary, characterizing smooth plane quartics with non-separable Gauss map.

\begin{Cor}\label{cor:fermat}
Let $C \subset \bp$ be a smooth plane quartic curve over an algebraically closed field $k$.
The following conditions are equivalent:
\begin{enumerate}
    \item \label{it:inff}
    the curve $C$ does not have finitely many inflection lines;
    \item \label{it:gm}
    the fibers of the Gauss map of $C$ have length at least $3$;
    \item \label{it:qh}
    the curve $C$ is totally harmonic;
    \item \label{it:fer}
    the characteristic of the field is $3$ and the curve $C$ is isomorphic to the Fermat curve $x^4+y^4+z^4 = 0$.
\end{enumerate}
\end{Cor}

\begin{proof}
Let $C^\vee \subset \bpd$ denote the image of the Gauss map of $C$.  By~\cite{kaji}*{Corollary~4.4}, the Gauss map is purely inseparable.
It follows that the curves $C$ and $C^\vee$ are birational and we deduce that $C^\vee$ has the same geometric genus as $C$.  In particular, $C^\vee$ is a plane curve of degree at least $4$.

\noindent
\eqref{it:inff}$\Longleftrightarrow$\eqref{it:gm}.
The Gauss map is separable if and only if a general tangent line to $C$ is not an inflection line, and the equivalence of \eqref{it:inff} and \eqref{it:gm} follows.

\noindent
\eqref{it:gm}$\implies$\eqref{it:qh}.
Since the fibers of the Gauss map $\gamma \colon C \to \bpd$ of $C$ have length at least $3$, every point of the dual curve $C^\vee$ corresponds to an inflection line and hence is contained in the vanishing set of the harmonic quartic $H (C)$.

\noindent
\eqref{it:qh}$\implies$\eqref{it:fer}.
The only smooth quartics in the list of Proposition~\ref{prop:nonsva} are projectively equivalent to the Fermat quartic curve in characteristic~$3$, as required.

\noindent
\eqref{it:fer}$\implies$\eqref{it:inff}.
Let $\ell$ be a tangent line to $C$, choose a parameterization of $\ell$ and denote by $F_\ell$ the restriction of the Fermat equation to the line $\ell$ under the chosen parameterization.  Since the harmonic quartic $H (x^4+y^4+z^4)$ vanishes, the equations in~\eqref{e:eval} imply that the invariant $S$ of $F_\ell$ vanishes.  Since $F_\ell$ has a repeated root, Remark~\ref{rem:orbite} shows that $F_\ell$ has at least a triple root and we conclude that $\ell$ is an inflection line for $C$ and we are done.
\end{proof}

A few of the implications of Corollary~\ref{cor:fermat} also follow from~\cite{pa}*{Proposition~3.7}.  In~\cite{fu}*{Theorems~1 and~2} there is a proof that, over fields of characteristic $3$, the only smooth plane quartics with degenerate Gauss map are projectively equivalent to the Fermat quartic.

\section{Configurations of inflection lines of plane quartics} \label{sec:trop}

Starting from here, unless specified otherwise, we assume that the characteristic of the field $k$ is coprime with~$6$.  We analyze the configurations of inflection lines associated to singular plane quartics and show that we can distinguish them from the configurations of inflection lines of general quartics.  We impose no restriction on the singularities of the quartic curve and the case that proves the hardest for us is the case of totally harmonic quartics.

\begin{Def}
Let $C \subset \bp$ be a plane quartic and let $\ell \subset \bp$ be a line not contained in $C$.  Suppose that the intersection $\ell \cap C$ contains a point $p$ with intersection multiplicity at least $3$ and such that $p$ is a smooth point of $C$.  We call $\ell$ a {\emph{simple inflection line}} if $p$ has multiplicity $3$; we call $\ell$ a {\emph{hyperinflection line}} if $p$ has multiplicity $4$.
\end{Def}

Remark~\ref{rem:flf} shows that, for every plane quartic $C$, simple inflection lines and hyperinflection lines correspond to points in the inflection scheme $\fl(C)$.  For an isolated point $\ell$ of the inflection scheme $\fl(C)$, we call {\emph{multiplicity of $\ell$}} in $\fl(C)$ the degree $\mult_\ell(\fl(C))$ of the irreducible component of $\fl(C)$ containing $\ell$. In the next lemma, we prove that, for general quartics, simple inflection lines and hyperinflection lines are characterized by their multiplicity in the inflection scheme.

\begin{Lem} \label{lem:singU}
Let $k$ be a field of characteristic coprime with~$6$.  Let $C \subset \bp$ be a
plane quartic curve and let $\ell \subset \bp$ be a line contained in $\fl(C)$.  Assume that $\ell$ is an isolated point of the inflection scheme $\fl(C)$.
\begin{enumerate}
\item
If the line $\ell$ is a simple inflection line, then
$\mult_\ell(\fl(C))$ equals~$1$.
\item
If the line $\ell$ is a hyperinflection line, then the harmonic quartic $H(C)$ is smooth at $\ell$, the harmonic sextic $K(C)$ is singular at $\ell$ and $\mult_\ell(\fl(C))$ equals~$2$.
\item
If there is a singular point of $C$ along $\ell$, then $\mult_\ell(\fl(C))$ is at least~$3$.
\end{enumerate}
Moreover, if $\ell$ is a hyperinflection line, then the tangent space to $H(C)$ at $\ell$ corresponds to the pencil of lines through the point $p$.
\end{Lem}

\begin{proof}
If the line $\ell$ is not contained in $C$, then let $p$ be the unique point of $\ell \cap C$ with intersection multiplicity at least $3$.  If the line $\ell$ is contained in $C$, then let $p$ be a singular point of $C$ on the line $\ell$.
Choose coordinates $x,y,z$ in $\bp$ so that the line $\ell$ is the line with equation $x=0$ and the point $p$ is the point $[0,0,1]$.  We write an equation of $C$ as
\[
q (x,y,z) = y^3 q_1(y,z) + x q_3(x,y,z) ,
\]
where $q_1$ and $q_3$ are forms of respective degrees $1$ and $3$.  Let $\alpha$ be the coefficient of the monomial $z^3$ in $q_3$; this coefficient vanishes if and only if the curve $C$ is singular at the point $[0,0,1]$.  Moreover, we also assume that
\begin{itemize}
    \item $q_1(y,z) = z$ if $\ell \cap C$ has intersection multiplicity exactly $3$ at $[0,0,1]$;
    \item $q_1(y,z) = y$ if $\ell \cap C$ has intersection multiplicity exactly $4$ at $[0,0,1]$;
    \item $q_1(y,z) = 0$ if $\ell$ is contained in $C$.
\end{itemize}

With our reductions, the harmonic quartic $H(q)$ and the harmonic sextic $K(q)$ vanish at the point $[1,0,0] \in \bpd$.  We compute the expansion of $H$ to first order near the point $[1,0,0]$ and the expansion of $K$ to at most second order near the same point.  By the definition of $H(q)$, we find
\[
H(q)(u,v,w) = u^4 S \left( y^3 q_1(y,z) + \left( - \frac{v}{u} y - \frac{w}{u} z \right) q_3 \left( - \frac{v}{u} y - \frac{w}{u} z , y , z \right) \right) .
\]
Using the definition of $S$, we obtain the congruences
\[
\begin{array}{l@{\hspace{30pt}}ll}
{\textrm{if }} q_1(y,z) = z, & {\textrm{ then }}
H (q) \equiv 3 \alpha u^3 v & \mod{ (v,w)^2 + (w)} ; \\[5pt]
{\textrm{if }} q_1(y,z) = y, & {\textrm{ then }}
H (q) \equiv -12 \alpha u^3 w & \mod{ (v,w)^2} .
\end{array}
\]
Analogously, we compute $K(q)$ and obtain the congruences
\[
\begin{array}{l@{\hspace{30pt}}ll}
{\textrm{if }} q_1(y,z) = z, & {\textrm{ then }}
K (q) \equiv 27 \alpha u^5 w & \mod{ (v,w)^2} ; \\[5pt]
{\textrm{if }} q_1(y,z) = y, & {\textrm{ then }}
K (q) \equiv -27 \alpha u^4 v^2
& \mod{ (v,w)^3 + (vw,w^2)} ; \\[5pt]
{\textrm{if }} q_1(y,z) = 0, & {\textrm{ then }}
K (q) \equiv 0 & \mod{ (v,w)^2} .
\end{array}
\]
Suppose that the curve $C$ is smooth at $p$.  This implies that the coefficient $\alpha$ does not vanish, that $\ell$ is not contained in $C$ by our choice of $p$, and that $\ell$ is either a simple inflection line or a hyperinflection line.  If $\ell$ is a simple inflection line, then the curves $H(C)$ and $K(C)$ are transverse at $[1,0,0]$ and hence the multiplicity of $\ell$ in $\fl(C)$ is $1$.  If $\ell$ is a hyperinflection line, then the curve $H(C)$ is smooth at $[1,0,0]$, the curve $K(C)$ has a double point and the tangent line to $H(C)$ at $[1,0,0]$ is not in the tangent cone to $K(C)$ at $[1,0,0]$.  In this case, the multiplicity of $\ell$ in $\fl(C)$ is $2$.  Observe that the tangent line to $H(C)$ at $\ell$ is the line with equation $w=0$ which corresponds to the pencil of lines through the point $p$.

Suppose that the curve $C$ is singular at the point $p$.  Thus, $\alpha$ vanishes and since the intersection multiplicity of $\ell$ and $C$ is at least $3$ at $p$, it follows that $\ell$ is in the tangent cone to $C$ at $p$.  By the computation of the harmonic sextic, the curve $K(C)$ has a point of multiplicity at least $3$ at $[1,0,0]$, as required.
\end{proof}

For fields of characteristic~0, \cite{PT1}*{Proposition~2.10} gives a more refined computation of the multiplicities than the one of Lemma~\ref{lem:singU}.

We give two examples showing that there is no direct link between the smoothness of a plane quartic $C$ and the smoothness of its harmonic quartic $H(C)$.

\begin{Exa} \label{ex:smsi}
{\emph{Smooth $C$ and singular $H(C)$.}}
The quartic in $\bp$ with equation $x^4 + y^4 + y z^3 = 0$ is smooth if the characteristic of $k$ is coprime with~$6$.  Its harmonic quartic $12 w (w^3 - v^3) = 0$ is the union of $4$ concurrent lines and in particular it is singular.
\end{Exa}

\begin{Exa} \label{ex:sism}
{\emph{Singular $C$ and smooth $H(C)$.}}
The quartic in $\bp$ with equation $x^2 y z + x y^3 + x z^3 + y^4 = 0$ has a node at $[1,0,0]$; its harmonic quartic $-12 u^3 w + 9 u^2 v w + 3 u v^3 + 3 u w^3 + v^2 w^2 = 0$ is smooth if the characteristic of $k$ is coprime with~$6$.
\end{Exa}

\begin{Rem}\label{rem:lirid}
In the hypotheses of Lemma~\ref{lem:singU}, if the curve $C$ is smooth, then the harmonic quartic $H(C)$ is reduced.  Indeed, every point of $\fl(C)$ is either a simple or a hyperinflection line and therefore corresponds to a smooth point of $H(C)$.  Since every irreducible component of $H(C)$ has at least one point in common with $K(C)$, we deduce that every irreducible component of $H(C)$ has smooth points, that is, $H(C)$ is reduced.
\end{Rem}

\begin{Prop} \label{prop:otto}
Let $k$ be a field of arbitrary characteristic, let $C \subset \bp$ be a smooth plane quartic and let $p$ be a point of $\bpd$.  The curve $C$ admits at most $8$ inflection lines through $p$, counted with multiplicity.
\end{Prop}

\begin{proof}
The only smooth curves in projective space admitting a point contained in every tangent line have degree at most~$2$ (see~\cite{Har}*{Theorem~IV.3.9}).  Thus, the projection of $C$ away from $p$ induces a finite separable morphism $\pi \colon C \to \mathbb{P}^1_k$ of degree at most~$4$.  Since $\pi$ is ramified at every inflection line of $C$ through $p$, we conclude using the Riemann-Hurwitz formula.
\end{proof}

\begin{Rem} \label{rem:hrid}
Let $C$ be a smooth plane quartic admitting exactly $8$ inflection lines through a point $p$, counted with multiplicity.  It follows that $C$ is projectively equivalent to $x^4 = y z (y-z) (y-\lambda z)$, where $\lambda \in k \setminus \{ 0,1 \}$ is a constant and $p$ is the point $[1,0,0]$.  Moreover, the inflection lines through $p$ are $4$ lines each of multiplicity $2$.
\end{Rem}

\begin{Prop}\label{prop:liscsing}
Let $k$ be a field of characteristic coprime with~$6$.  Let $C \subset \bp$ be a singular plane quartic not containing a line with multiplicity at least three.  Every configuration of inflection lines associated to $C$ is not the configuration of inflection lines of a smooth quartic.
\end{Prop}

\begin{proof}
Let $D \subset \bp$ be a smooth quartic curve and suppose by contradiction that $\fl(D)$ is a configuration of inflection lines associated to $C$.  We use the following three properties to obtain a contradiction:
\begin{itemize}
\item
the inflection scheme $\fl(D)$ is a complete intersection of dimension zero (Lemma~\ref{lem:defU}) and with irreducible components of multiplicity at most $2$ (Lemma~\ref{lem:singU});
\item
the quartic $H(D)$ is reduced (Remark~\ref{rem:lirid});
\item
the configuration $\fl(D)$ does not consist of concurrent lines (Proposition~\ref{prop:otto}).
\end{itemize}

We begin assuming that $C$ is reduced with singular points of multiplicity~$2$.  In this case, by Lemma~\ref{lem:defU} the inflection scheme $\fl(C)$ has dimension zero and is the unique configuration of inflection lines associated to $C$.  We apply Lemma~\ref{lem:singU} to a line $\ell$ in the tangent cone to a singular point of $C$ to deduce that $\mult_\ell(\fl(C))$ is at least~$3$.  Thus, $\fl(C)$ and $\fl(D)$ are different.

We are left with the cases in which either $C$ is a double conic or $C$ has a singular point of multiplicity at least~$3$.  Denote by $q_C$ and $q_D$ ternary quartic forms vanishing on $C$ and $D$ respectively.  Since $\fl(D)$ is the complete intersection with ideal generated by $H(q_D),K(q_D)$ and $H(q_C)$ is in the ideal, it follows that $H(q_C)$ is a multiple of $H(q_D)$.  In particular, if $C$ is not totally harmonic, then $H(C)$ and $H(D)$ coincide and hence $H(C)$ is reduced.

Suppose that $C$ is an irreducible conic with multiplicity two.  Choose coordinates on $\bp$ so that $q_C$ is the polynomial $(xy-z^2)^2$.  Computing $H(q_C)$ we obtain $(4uv-w^2)^2$, contradicting the reducedness of $H(C)$.

In the remaining cases, $C$ has a singular point $p$ of multiplicity at least~$3$.  Choose coordinates on $\bp$ so that $p$ is the point $[1,0,0]$.  Thus, there are binary forms $q_3,q_4$ of degrees $3$ and $4$ respectively such that $q_C$ equals $x q_3(y,z) + q_4(y,z)$.  Using the definition of $H$ it is clear that $H(q_C)$ has degree at most $2$ as a polynomial in the variables $v,w$.  Therefore, the harmonic quartic $H(q_C)$ is divisible by $u^2$.  If $C$ is not totally harmonic, then we are done, since $H(C)$ is non-reduced in this case.  Otherwise, $C$ is totally harmonic and using Proposition~\ref{prop:nonsva} and the assumption that $C$ does not contain a line with multiplicity at least~$3$, it follows that $C$ consists of~$4$ distinct lines through the point $p$.  Every line in $\fl(D)$ must therefore contain the point $p$ and hence the configurations of inflection lines associated to $C$ consist of~$24$ lines through the point $p$, counted with multiplicity.  This is impossible by Proposition~\ref{prop:otto}.
\end{proof}

In the next proposition, we obtain properties of the configuration of inflection lines in the cases missing from the statement of Proposition~\ref{prop:liscsing}, namely, plane quartics containing a line with multiplicity at least~$3$.

Let $\wtD$ be the closure in $\Pq \times \Hilb \times \Pqd$ of the locus
\[
\wtD = \bigl\{(C, \fl(C) , H(C)) : C \text{ is a general plane quartic} \bigr\}.
\]
Let $(C,\fl)$ be a pair in $\wtP$.  We say that a quartic $D \subset \bpd$ is {\emph{{\nom} for $(C,\fl)$}} if the triple $(C,\fl,D)$ is in $\wtD$.  It follows from the definitions that if $D$ is derived for $(C,\fl)$, then $D$ contains $\fl$.

\begin{Lem}\label{lem:syslin}
Let $(C,\fl)$ be a pair in $\wtP$ and let $\mathscr{L} \subset \Pqd$ denote the linear system of quartics containing $\fl$.  There is at least one {\nom} quartic in $\mathscr{L}$.  If $\mathscr{L}$ consists of more than one elements, then every quartic in $\mathscr{L}$ is not integral.
\end{Lem}

\begin{proof}
The first part is clear from the properness of $\wtD$.  Suppose that $D,E$ are distinct elements of $\mathscr{L}$.  We proceed by contradiction and assume that $D$ is integral.  Since $D$ and $E$ are distinct quartics and $D$ is integral, we deduce that the intersection $D \cap E$ has dimension~$0$, degree~$16$ and contains $\fl$.  This is impossible, since the degree of $\fl$ is~$24$.  Thus, $D$ cannot be integral and we are done.
\end{proof}

Let $R$ be a DVR with residue field $k$ and maximal ideal $\mathfrak{m}$ generated by $t \in \mathfrak{m}$.  We define a function $r \colon R[x_1 , \ldots , x_n] \to k[x_1 , \ldots , x_n]$ that we call {\emph{$k$-reduction}}, as follows.  Let $f \in R[x_1,\ldots,x_n]$ be a polynomial.  If $f$ vanishes, then we define $r(f) = 0$.  If $f$ is non-zero, then let $v$ denote the largest power of $t$ dividing all the coefficients of $f$ and write $f = t^v f'$ with $f' \in R[x_1 , \ldots , x_n]$.  By construction, $f'$ is non-zero modulo $\mathfrak{m}$, and we define $r(f) \in k[x_1 , \ldots , x_n]$ to be the reduction of $f'$ modulo $\mathfrak{m}$.

\begin{Prop} \label{prop:re3}
Let $k$ be a field of characteristic coprime with~$6$.  Let $(C,\fl)$ be a pair in $\wtP$ and suppose that the quartic $C \subset \bp$ contains a line $\ell$ with multiplicity at least $3$.  Every {\nom} quartic for $(C,\fl)$ is singular at the point corresponding to the line $\ell$.
\end{Prop}

\begin{proof}
Let $D$ be a {\nom} quartic for $(C,\fl)$.  Choose coordinates $x,y,z$ on $\bp$ so that the line $\ell$ has equation $x=0$ and $C$ is either $x^4$ or $x^3 y$.  Choose a smooth irreducible curve $B$ and a morphism $B \to \wtD$ such that
\begin{itemize}
\item
there is a point $0 \in B$ with $0 \mapsto (C,\fl,D)$,
\item
the image of the generic point of $B$ is the triple $(C',\fl_{C'},H(C'))$, where $C'$ is a non-totally harmonic smooth quartic curve on which $\F$ is defined.
\end{itemize}
Denote by $R$ the local ring of $B$ at $0$.  Since the curve $B$ is smooth, the ring $R$ is a DVR.

Let $q(x,y,z) = \sum a_{ijk} x^i y^j z^k \in R[x,y,z]_4$ be a ternary quartic form defining the curve associated to the generic point of the image of $B$.  By the definition of $B$, the $k$-reduction $r(q)$ of $q$ is a form defining the curve $C$ and the $k$-reduction of $H(q)$ is a form defining the curve $D$.  It follows that $r(q)$ is proportional to either $x^4$ or $x^3y$ and thus, rescaling $q$ if necessary, we can assume that all the coefficients of $q$ are in $R$ and that one among $a_{400},a_{310}$ equals~$1$, so that all remaining ones are not invertible.  Write the harmonic quartic form of $q$ as $H(q)(u,v,w) = \sum h_{ijk} u^i v^j w^k \in R[u,v,w]$.

Denote by $I_2$ the set of exponents $(i,j,k)$ of monomials of degree $4$ in $3$ variables with $i \leq 2$.  Let $\mathfrak{a} \subset R$ denote the ideal generated by the coefficients of $q$ that are not invertible; let $\mathfrak{h} \subset R$ denote the ideal generated by the coefficients $h_{ijk}$ of $H(q)$ for which $(i,j,k)$ is in $I_2$.

If $x^4$ is an equation of $C$, then the identities 1,~2,~3 of Corollary~\ref{cor:lie} become
\[
\begin{array}{lrrrl}
1: &
48 h_{400} & + 6 a_{310} h_{310} & + 9 a_{301} h_{301} & = -
\displaystyle
\sum_{(i,j,k) \in I_2}
(i-j) \frac{i!j!k!}{2} a_{ijk} h_{ijk}
\\[13pt]
2: &
& 12 h_{310} & & = -
\displaystyle
\sum_{(i,j,k) \in I_2}
(i+1) \frac{i!j!k!}{2} a_{(i+1)(j-1)k} h_{ijk}
\\[13pt]
3: &
& & 12 h_{301} & = -
\displaystyle
\sum_{(i,j,k) \in I_2}
(i+1) \frac{i!j!k!}{2} a_{(i+1)j(k-1)} h_{ijk}.
\end{array}
\]

If $x^3y$ is an equation of $C$, then the identities 4,~5,~6 of Corollary~\ref{cor:lie} become
\[
\begin{array}{lrrrl}
4: &
12 h_{400} & + 6 a_{220} h_{310} & + 3 a_{211} h_{301} & = -
\displaystyle
\sum_{(i,j,k) \in I_2}
(j+1) \frac{i!j!k!}{2} a_{(i-1)(j+1)k} h_{ijk}
\\[13pt]
5: &
& 3 h_{310} & - 3 a_{301} h_{301} & = -
\displaystyle
\sum_{(i,j,k) \in I_2}
(j-k) \frac{i!j!k!}{2} a_{ijk} h_{ijk}
\\[13pt]
6: &
& & 3 h_{301} & = -
\displaystyle
\sum_{(i,j,k) \in I_2}
(j+1) \frac{i!j!k!}{2} a_{i(j+1)(k-1)} h_{ijk}.
\end{array}
\]
In either case, we obtain that the coefficients $h_{400}, h_{310}, h_{301}$ of $H(q)$ lie in the product $\mathfrak{a} \mathfrak{h}$ of the ideals~$\mathfrak{a}$ and~$\mathfrak{h}$.  Recall that $H(q)$ does not vanish.  Denote by $\rho$ the minimum valuation of a coefficient of $H(q)$: the minimum $\rho$ is attained at a monomial $h_{ijk}$ with $(i,j,k) \in I_2$.  Moreover, the valuation $\rho$ is strictly smaller than the valuation of each of the three coefficients $h_{400}, h_{310}, h_{301}$.  As a consequence, the $k$-reduction $r(H(q)) \in k[u,v,w]$ does not involve any of the monomials $u^4 , u^3 v , u^3 w$.  We conclude that the quartic $D \colon r(H(q)) = 0$ is singular at the point $[1,0,0]$ of $\bpd$ corresponding to the line $\ell$, as required.
\end{proof}

We want to give examples showing that the conclusion of Proposition~\ref{prop:re3} does not hold if the characteristic of the field $k$ divides~$6$.  To construct these examples, we prove a lemma, where we lift our standing assumption that the characteristic of $k$ is coprime with~$6$.  Again, we denote by $R$ a DVR and $k$ its residue field.

\begin{Lem} \label{lem:acca0}
Let $q(x,y,z) \in R[x,y,z]_4$ be a non-zero ternary quartic form.  Let $C_0 \subset \bp$ be the quartic defined by the vanishing of the $k$-reduction $r(q)$ of $q$ and let $H_0 \subset \bpd$ be the quartic defined by the vanishing of $r(H(q))$.  Then, there is a configuration $\fl_0$ of inflection lines associated to $C_0$ and $H_0$ is {\nom} for $(C_0,\fl_0)$.
\end{Lem}

\begin{proof}
Denote by $F$ the field of fractions of $R$.  Let $q_F$ be the image of $q$ under the inclusion $R[x,y,z] \subset F[x,y,z]$ and let $C_F \subset \mathbb{P}^2_F$ be the quartic curve with equation $q_F(x,y,z)=0$.  Choose a pair $(C_F,\fl)$ in $\widetilde{\mathbb{P}}^{14}_F$, so that $\fl$ is a configuration of inflection lines associated to $C_F$.  Since the configuration $\fl$ is contained in the inflection scheme $\fl(C_F) = H(C_F) \cap K(C_F)$, we deduce that the form $H(q_F)$ vanishes on $\fl$.  By properness of the Hilbert scheme, the configuration $\fl \in {\rm{Hilb}}_{24} \left( \mathbb{P}^2_F \right)$ admits a specialization to a configuration of inflection lines $\fl_0$ associated to $C_0$.  Under the inclusion $R[u,v,w] \subset F[u,v,w]$ the form $H(q)$ corresponds to the form $H(q_F)$.  Therefore, the quartic $H_0$ defined by $r(H(q)) = 0$ is derived for the pair $(C_0,\fl_0)$, as needed.
\end{proof}

\begin{Exa}
Let $k$ be a field, let $R$ denote the local ring of $\mathbb{A}^1_k$ at the origin $0$ and let $t$ be a local parameter near $0$.  We apply Lemma~\ref{lem:acca0} to the forms $q$ in $R[x,y,z]_4$ appearing in Table~\ref{tab:limiti}.  By Lemma~\ref{lem:acca0}, all the $k$-reductions $r(H(q))$ in Table~\ref{tab:limiti} vanish on a configuration of inflection lines associated to the plane quartic with equation $r(q)=0$.  In our examples, the quartic with equation $r(q)=0$ contains the line $x=0$ with multiplicity at least~$3$; nevertheless, the quartics with equation $r(H(q))=0$ do not contain the point $[1,0,0]$. Over a field of characteristic~$2$, the $k$-reductions $r(H(q))$ define smooth quartics;  over a field of characteristic~$3$, the $k$-reductions $r(H(q))$ define the double of a smooth conic.
\begin{table}[h]
    \centering
    $\begin{array}{|c|c|c|}
\hline
\textrm{Form } q(x,y,z) &
{\textrm{Characteristic of }} k &
\textrm{$k$-reduction } r(H(q)) (u,v,w) \\
\hline
\vphantom{W^{W^{W^W}}}
x^4+t (x^3 y + y^3 z + z^3 x + z^3 y) & 2 &
u^4 + u^3 v + v^3 w + w^3 u + u v^2 w \\[5pt]
x^3 y + t(y^2 z^2 + t (x z^3 + y^3 z)) & 2 &
u^4 + u w^3 + v^3 w \\[5pt]
x^4 + t (x^2 y z + y^2 z^2) & 3 &
(u^2 - v w)^2 \\[5pt]
x^3 y + t (x^2 y z + y^2 z^2) & 3 &
(u^2 - v w)^2 \\[3pt]
\hline
    \end{array}$
    \caption{Some quartic forms and $k$-reductions of their harmonic quartics.}
    \label{tab:limiti}
\end{table}
\end{Exa}

\begin{Exa}
Let $k$ be a field of characteristic coprime with~$6$. Let $R$ denote the local ring of $\mathbb{A}^1_k$ at the origin $0$ and let $t$ be a local parameter near $0$. Let $C \subset \bp$ be the (reducible) plane quartic with equation $y^4 - y z^3 = 0$; the quartic $C$ is, up to projective equivalence, the only totally harmonic quartic not containing a line with multiplicity at least 3.  We apply Lemma~\ref{lem:acca0} to the form $q = (y^4 - y z^3) + t (x^4 + z^4)$ in $R[x,y,z]_4$. We obtain that the $k$-reduction $r(H(q)) = u^4 + v^3 w + w^4$ defines a smooth quartic containing a configuration of inflection lines associated to $C$.
\end{Exa}

\section{The reconstruction of the general plane quartic}
\label{sec:rifi}

In this section, we assume that the field $k$ is algebraically closed and that its characteristic is coprime with~$6$.  We prove our main reconstructibility result in Theorem~\ref{thm:rf}.

Let $\varepsilon \in k$ satisfy $\varepsilon^4 = 1$ and let $\qep$ be the quartic form
\[
\qep = (x-y)^4 + (x-z)^4 + (y-\varepsilon z)^4 - (x^4 + y^4 + z^4) .
\]
Observe that the three identities
\[
\qev (x,y,0) = (x-y)^4
\quad \quad
\qev (x,0,z) = (x-z)^4
\quad \quad
\qev (0,y,z) = (y- \varepsilon z)^4
\]
hold.

\begin{Lem}\label{lem:conc}
A smooth plane quartic $C$ has $3$ non concurrent hyperinflection lines if and only if, up to a change of coordinates, there are constants $\varepsilon , \lambda , \mu , \nu \in k$ with $\varepsilon^4 = 1$ such that
\[
\qep + xyz (\lambda x + \mu y + \nu z) = 0
\]
is an equation of $C$.
\end{Lem}

\begin{proof}
Suppose first that $C$ is a smooth plane quartic having the three lines $\ell_1 , \ell_2 , \ell_3$ as non concurrent hyperinflection lines.  Changing coordinates in $\bp$ if necessary, we reduce to the case in which the three lines have equations
\[
\ell_1 \colon x=0
\quad \quad \quad
\ell_2 \colon y=0
\quad \quad \quad
\ell_3 \colon z=0 .
\]
For $i \in \{1,2,3\}$ let $p_i \in C$ denote the hyperinflection point of $C$ corresponding to the hyperinflection line $\ell_i$: none of the three points $p_1,p_2,p_3$ is a coordinate point, since otherwise two coordinate lines would be tangent to $C$ at the same point.  Rescaling the $x$ and the $y$ coordinate if necessary, we reduce further to the case in which $p_3 = [1,1,0]$ and $p_2 = [1,0,1]$.  Write $p_1 = [0,1,\epsilon]$ for the remaining hyperinflection point, where $\epsilon \in k$.

Let $q(x,y,z)$ be a ternary quartic form.  We now impose the linear conditions implying that the coordinate lines are hyperinflection lines to the plane quartic defined by $q(x,y,z) = 0$ at the points $p_1,p_2,p_3$.  Restricting the polynomial $q$ to each line $z=0$, $y=0$ and $x=0$ in succession we find the conditions
\[
q(x,y,0) = \alpha (x-y)^4
\quad \quad \quad
q(x,0,z) = \beta (x-z)^4
\quad \quad \quad
q(0,y,z) = \gamma (y-\epsilon z)^4 ,
\]
where $\alpha , \beta , \gamma \in k$ are non-zero constants.  Moreover, $\alpha$ and $\beta $ must coincide, since they are the coefficient of $x^4$ in $q$.  Similarly, $\alpha $ and $\gamma$ must also coincide, since they are the coefficient of $y^4$ in $q$.  Analogously, $\beta $ and $\gamma \epsilon^4$ must also coincide, since they are the coefficient of $z^4$ in $q$.  Rescaling $q$ if necessary, we reduce to the case in which the identities $\alpha = \beta = \gamma = \epsilon^4 = 1$ hold.

Finally, observe that the only remaining coefficients of $q$ are the coefficients of the three monomials $x^2yz , xy^2z , xyz^2$, as required.

The converse is clear.
\end{proof}

Lemma~\ref{lem:conc} gives a standard form for a smooth plane quartic with the coordinate lines as hyperinflection lines.  In the next lemma, we determine some of the conditions arising from imposing a further hyperinflection line.

\begin{Lem} \label{lem:mab}
Let $\qep + xyz (\lambda x + \mu y + \nu z) = 0$ be the equation of a quartic curve $C$, let $a,b \in k$ be constants and let $\ell$ be the line with equation $z = ax + by$.  If $C$ and $\ell$ meet at a unique point $p$, then there are constants $d_p,e_p,f_p \in k$, depending on $p$, such that the identity
\[
\begin{pmatrix}
a & 0 & a^2 \\
b & a & 2ab \\
0 & b & b^2
\end{pmatrix}
\begin{pmatrix}
\lambda \\ \mu \\ \nu
\end{pmatrix} =
\begin{pmatrix}
d_p \\ e_p \\ f_p
\end{pmatrix}
\]
holds.
\end{Lem}

\begin{proof}
Restricting the equation of $C$ to the line $\ell$ we find the binary quartic form
\[
c_\ell (x,y) = \qev (x,y,ax+by) + xy(ax+by) (\lambda x + \mu y + \nu (ax+by)) .
\]
By assumption, the form $c_\ell$ is the fourth power of a non-zero linear form.  This implies that at least one of the coefficients of $x^4$ and $y^4$ is non-zero.  Exchanging the roles of $x$ and $y$ if necessary, we reduce to the case in which the coefficient $(1-a)^4$ of $x^4$ is non-zero and hence there is a constant $\pi \in k$ such that $c_\ell = (1-a)^4 (x-\pi y)^4$.  Equating the coefficients of $x^3y , x^2y^2 , xy^3$ we obtain the required relations.
\end{proof}

\begin{Prop} \label{prop:cinf}
Let $p_1 , \ldots , p_5 \in \bp$ be distinct points and let $\ell_1 , \ldots , \ell_5 \subset \bp$ be distinct lines such that for $i \in \{1 , \ldots , 5\}$ the point $p_i$ lies on the line $\ell_i$.  There is at most one smooth quartic curve $C$ admitting $\ell_1 , \ldots , \ell_5$ as hyperinflection lines and the points $p_1 , \ldots , p_5$ as corresponding hyperinflection points.
\end{Prop}

\begin{proof}
Suppose that $C$ is a smooth quartic curve with the required properties.  First, the five lines $\ell_1 , \ldots , \ell_5$ cannot contain a common point by Proposition~\ref{prop:otto}.  Thus, there are three non concurrent lines among $\ell_1 , \ldots , \ell_5$; relabeling the lines, if necessary, we assume that $\ell_1 , \ell_2 , \ell_3$ are non concurrent.  Applying Lemma~\ref{lem:conc} to the quartic $C$ and the lines $\ell_1 , \ell_2 , \ell_3$, we obtain coordinates on $\bp$ such that an equation of $C$ is
\[
C \colon \quad \qep + xyz (\lambda x + \mu y + \nu z) = 0 .
\]
The two hyperinflection lines $\ell_4$ and $\ell_5$ are therefore not coordinate lines.  As a consequence, each of these two lines has an equation involving at least two of the variables $x,y,z$.  Permuting if necessary the variables, we reduce to the case in which the equations of the two lines $\ell_4$ and $\ell_5$ involve the variable $z$: let $z = ax + by$ be an equation for $\ell_4$ and $z = cx + dy$ an equation for $\ell_5$.  Applying Lemma~\ref{lem:mab} we find that the matrix $M_{abcd}$ of the not necessarily homogeneous linear system satisfied by the coefficients $\lambda , \mu , \nu$ of the equation of $C$ is
\[
M_{abcd} =
\begin{pmatrix}
a & 0 & a^2 \\
b & a & 2ab \\
0 & b & b^2 \\
c & 0 & c^2 \\
d & c & 2cd \\
0 & d & d^2
\end{pmatrix} .
\]
It is now easy to check that the matrix $M_{abcd}$ has rank $3$ as soon as $(a,b)$ and $(c,d)$ are distinct and different from $(0,0)$.  We conclude that there is at most one solution to the system, as required.
\end{proof}

\begin{Cor}\label{cor:rico}
Let $C,D$ be smooth quartic curves.  If $\fl(C)$ and $\fl(D)$ coincide and $C$ has at least $5$ hyperinflection lines, then $C$ and $D$ coincide.
\end{Cor}

\begin{proof}
The two harmonic quartics $H(C)$ and $H(D)$ meet at the points corresponding to inflection lines.  Moreover, at the points corresponding to hyperinflection lines, they also have a common tangent direction, thanks to Lemma~\ref{lem:singU}.  This implies that $D$ must also go through the hyperinflection points of $C$.  Since $C$ has at least $5$ hyperinflection lines and $D$ shares these same lines as hyperinflection lines and also goes through the corresponding hyperinflection points, we conclude using Proposition~\ref{prop:cinf} that $C$ and $D$ coincide, as required.
\end{proof}

In the next lemma, we study a pencil $V_t$ of smooth plane quartic curves with at least $8$ hyperinflection lines.  Over fields of characteristic zero, the pencil $V_t$ was studied by Vermeulen~\cite{V} and Girard-Kohel~\cite{GK}; we also used this pencil in~\cite{PT2}.  For properties of the pencil $V_t$, see the cited references.

\begin{Lem}\label{lem:Vt}
Let $k$ be a field of characteristic coprime with~$6$, let $t$ be a constant in $k$ and let $V_t$ be the plane quartic with equation $t x^4 + y^4 - z^4 - 2 x^2y^2 -4 xyz^2=0$.  For $t$ different from $0,1,\frac{1}{81}$ the quartic $V_t$ is smooth and has exactly $8$ hyperinflection lines and exactly $8$ simple inflection lines.  The harmonic quartic of $V_t$ is
\[
H(V_t) \colon \quad
-3 u^4 - 3 t v^4 + (3 t + 1) w^4 + 10 u^2 v^2 + 8 u v w^2 = 0 ;
\]
for general $t$, the curve $H(V_t)$ is smooth.  More precisely, $H(V_{-1})$ is smooth.
\end{Lem}

\begin{proof}
It is easy to check that $V_t$ is smooth for $t$ different from $0,1$ and that the hyperinflection points of the curve $V_t$ are the intersection of $V_t$ with the two lines $x=0$ and $y=0$.  The remaining inflection points are the points of intersection of the curve $V_t$ with the conic $2xy+3z^2 = 0$.  These remaining intersection points satisfy the equations
\[
2xy+3z^2 = 0
\quad \quad \quad
{\textrm{and}}
\quad \quad \quad
t x^4 + \frac{2}{9} x^2y^2 + y^4 = 0 .
\]
The discriminant of the polynomial $t x^2 + \frac{2}{9} x y + y^2$ is $4(\frac{1}{81}-t)$.  Thus, for $t$ different from $\frac{1}{81}$ and $0$, there are $8$ distinct inflection points on $V_t$, besides the $8$ hyperinflection points already determined.  We deduce that for the values of $t$ not equal to $0,1,\frac{1}{81}$, the curve $V_t$ is smooth and has exactly $8$ hyperinflection lines and $8$ hyperinflection points.

Evaluating the harmonic is an easy computation as is checking that $H(V_{-1})$ is smooth for all fields of characteristic coprime with~$6$.
\end{proof}

The curves in the following remark appear in the main result of~\cite{PT1}.

\begin{Rem} \label{rem:hilsym}
Let $k$ be an algebraically closed field of characteristic~$13$ and let $V_{-1}$ and $V'_{-1}$ be the smooth plane quartics with equations
\[
\begin{array}{llrcl}
V_{-1} & \colon & \quad \quad
- x^4 + y^4 - z^4 - 2 x^2y^2 -4 xyz^2 & = & 0 , \\[5pt]
V'_{-1} & \colon & \quad \quad
x^4 - y^4 - z^4 - 2 x^2y^2 -4 xyz^2 & = & 0 .
\end{array}
\]
These curves are equivalent under the exchange of the coordinates $x$ and $y$.  The configurations of inflection lines associated to these two curves are different points in $\Hilb$: they have distinct and smooth harmonic quartic curves.  Nevertheless, these two configurations have the same image under the Chow morphism $\Hilb \to \Sym$ (see~\cite{PT1}*{Theorem on p.~2}).  In fact, denote by $\ell_1 , \ldots , \ell_8$ the hyperinflection lines of $V_{-1}$ and by $\ell_9 , \ldots , \ell_{16}$ the simple inflection lines of $V_{-1}$.  The lines $\ell_1 , \ldots , \ell_8$ are also the hyperinflection lines of $V'_{-1}$ and the lines $\ell_9 , \ldots , \ell_{16}$ are also the simple inflection lines of $V'_{-1}$.  The configuration $\fl(V_{-1})$ consists of~$8$ non-reduced points corresponding to the lines $\ell_1 , \ldots , \ell_8$ and~$8$ reduced points corresponding to $\ell_9 , \ldots , \ell_{16}$.  The same is true for $\fl(V'_{-1})$, but the non-reduced structure at each of the~$8$ lines $\ell_1 , \ldots , \ell_8$ is different (we can check this using Lemma~\ref{lem:singU}).  As points in $\Hilb$, the configurations $\fl(V_{-1})$ and $\fl(V'_{-1})$ are different.  As zero-cycles, they are both represented by the same cycle $2(\ell_1 + \cdots + \ell_8) + (\ell_9 + \cdots + \ell_{16})$.
\end{Rem}

For every line $\ell\subset \bp$, we define the loci
\begin{eqnarray*}
\Vinf & := & \{C\in \Pq :  \ell\cdot C=3p+q, \text{ for some } p,q\in C\} , \\[5pt]
\VH & := & \{C\in \Pq :  [\ell] \in H(C) \subset \bpd \} .
\end{eqnarray*}
The locus $\Vinf$ is the codimension $2$ locus of quartics in $\Pq$ having the line $\ell$ as an inflection line.  The locus $\VH$ is the codimension $1$ locus of quartics in $\Pq$ whose harmonic quartic contains the point corresponding to the line $\ell$; equivalently, $\VH$ is the locus of quartics in $\Pq$ whose restriction to the line $\ell$ is a binary quartic form with vanishing invariant $S$.  Clearly, there is an inclusion $\Vinf \subset \VH$, since the restriction of the equation of a quartic curve to an inflection line is a binary form with a triple root and hence vanishing invariant $S$.  Denote by $C$ a smooth plane quartic having $\ell$ as an inflection line and let $T_C \Vinf$, $T_C \VH$ and $T_C \Pq$ denote the Zariski tangent spaces to $\Vinf$, $\VH$ and $\Pq$, respectively, at the point corresponding to $C$.  If the line $\ell$ is a simple inflection line, then $\Vinf$ is smooth at $C$ (Lemma~\ref{lem:tang}) and hence $T_C \VH$ has codimension~$2$ in $T_C \Pq$.  Suppose that $\ell$ is a hyperinflection line for $C$.  The locus $\Vinf$ may be singular at the point corresponding to $C$ (it is easy to see that this is indeed the case).  Nevertheless, the inclusions $\Vinf \subset \VH \subset \Pq$ induce inclusions at the level of Zariski tangent spaces $T_C \Vinf \subset T_C \VH \subset T_C \Pq$.  Hence, $T_C \Vinf$ is contained in $T_C \VH$ and $\VH$ is smooth at $C$ (Lemma~\ref{lem:tang}).  We deduce that $T_C\Vinf$ has codimension at least~$1$ in $T_C \Pq$.  In this way, even at some singular points of $\Vinf$, we obtain a linear condition imposed by $\ell$ on $T_C \Vinf$.  The following lemma gives a cohomological interpretation of these tangent spaces.

\begin{Lem} \label{lem:tang}
Let $k$ be a field of characteristic different from $3$.  Let $C$ be a smooth plane quartic, let $\ell$ be an inflection line of $C$ and write $\ell\cdot C = 3p+q$, for $p,q\in C$.  The following tangent space computations hold:
\[
\begin{array}{rcll}
T_{C}\Vinf & \simeq & \HH(C,\sO_C(4)\otimes \sO_C(-2p)) , & \textrm{if } p \neq q; \\[5pt]
T_{C}\VH & \simeq & \HH(C,\sO_C(4)\otimes \sO_C(-p)) , & {\textrm{if $\car (k) \neq 2$ and }} p = q.
\end{array}
\]
\end{Lem}

\begin{proof}
For the computation of the tangent space to $\Vinf$ see~\cite{PT1}*{Lemma~3.1} (although the stated reference assumes characteristic $0$, the proof of the statement only requires characteristic different from~$3$).  We still need to compute the tangent space to $\VH$ at $C$ when $\ell$ is a hyperinflection line of $C$.  Choose coordinates on $\bp$ so that $\ell$ is the line $x=0$ and the point $p$ is the point $[0,0,1]$.  It follows that a polynomial $q_C$ defining the quartic $C$ is $q_C = y^4 + x c(x,y,z)$, where $c$ is a ternary cubic form.  To each ternary quartic form $q \in \HH (\bp , \sO_{\bp} (4))$, we associate a first order deformation of the quartic curve $C$ by $q_C + \varepsilon q$, where $\varepsilon^2 = 0$.  In this way, we identify the tangent space to $\Pq$ at the quartic $C$ with $\HH(C,\sO_C(4))$.  The first order deformations tangent to $\VH$ therefore are the ones for which the invariant $S((q_C + \varepsilon q)|_\ell)$ vanishes.  This condition translates to
\[
S(y^4 + \varepsilon q(0,y,z)) = \varepsilon \bilq{y^4}{q(0,y,z)} = \varepsilon \bilq{y^4}{q(0,0,z)} = 0 .
\]
Letting $\gamma$ be the coefficient of $z^4$ in $q$, this condition translates to the equation $12 \gamma = 0$.  We deduce that $\gamma $ vanishes and hence $q$ vanishes at the hyperinflection point $[0,0,1]$, as needed.
\end{proof}

Let $\usm \subset \Pq$ denote the open subset consisting of smooth plane quartics $C$ with at most $10$ hyperinflection lines. By Lemma \ref{lem:defU}, the rational map $\F$ is defined at $\usm$.

\begin{Lem}\label{lem:et}
The morphism $\F|_{\usm} \colon \usm \to \Hilb$ is unramified.
\end{Lem}

\begin{proof}
Let $C \in \usm$ be a smooth plane quartic with at most $10$ hyperinflection lines and let $F_C = (\F|_{\usm})^{-1} \F(C)$ be the fiber of $\F|_{\usm}$ containing $C$.  Let $p_1 , \ldots , p_{f+h}$ denote the inflection points of $C$ and let $\ell_1 , \ldots , \ell_{f+h}$ denote the corresponding inflection lines.  We label the points in such a way that for $i \in \{1, \ldots , f+h\}$ we have $\ell_i\cdot C = 4p_i$ if and only if $i \geq f+1$.  Let $\Delta_C$ denote the divisor $\Delta_C = \sum_{i=1}^f 2 p_i + \sum_{j=1}^h p_{f+j}$ on $C$.  By the discussion before Lemma~\ref{lem:tang}, the sequence of inclusions
\[
F_C
\quad \subset \quad
\Vinfl{1} \cap \cdots \cap \Vinfl{f+h}
\quad \subset \quad
\left( \mathop{\bigcap}_{i = 1}^{f} \Vinfl{i} \right)
\cap
\left( \mathop{\bigcap}_{j = 1}^{h} \VHl{f+j} \right)
\]
holds.  Therefore, by Lemma~\ref{lem:tang}, the tangent space to $F_C$ at $C$ is contained in the vector space $\HH \bigl( C,\sO_C( 4L - \Delta_C \bigr)$, where $L$ is the class on $C$ of a line in $\bp$.  By the assumption that $C$ has at most $10$ hyperinflection lines, the degree of the divisor $4L-\Delta_C$ is $16-(2f+h) \leq -2$ and hence the tangent space to the fiber $F_C$ at $C$ is zero.
\end{proof}

\begin{Thm} \label{thm:rf}
Let $k$ be a field of characteristic coprime with~$6$.  The morphism $\wtF$ is birational onto its image.
\end{Thm}

\begin{proof}
Let $C$ be a plane quartic with the following properties:
\begin{itemize}
\item
$C$ is smooth;
\item
$H(C)$ is smooth;
\item
$C$ has at least $5$ and at most $10$ hyperinflection lines.
\end{itemize}
Such quartics exist: the quartic $V_{-1}$ in the family of Lemma~\ref{lem:Vt} is an example.

The result follows if we show that the morphism $\F$ is unramified at the curve $C$ and that $C$ is reconstructible from $\fl(C)$.  Since $C$ is smooth and has at most $10$ hyperinflection lines, Lemma~\ref{lem:et} allows us to conclude that $\F$ is unramified at $C$.

To finish the proof, we show that $C$ is reconstructible from $\fl(C)$.  Let $D$ be a plane quartic and suppose that $\fl(C)$ is a configuration of inflection lines associated to $D$.  Equivalently, this means that the pair $(D,\fl(C))$ is in $\wtP$. Assume first that $D$ is singular.  Proposition~\ref{prop:liscsing} shows that $D$ must contain a line with multiplicity at least~$3$.  Proposition~\ref{prop:re3} then shows that every {\nom} quartic containing $\fl(C)$ must be singular, contradicting the assumption that $H(C)$ is non-singular.  Thus, $D$ must be smooth.  Since $C$ has at least $5$ hyperinflection lines, Corollary~\ref{cor:rico} shows that $C$ and $D$ coincide.  Hence, $C$ is reconstructible from $\fl(C)$ and we are done.
\end{proof}

\end{document}